\documentclass[aop,MSNbibl,nameyear,dvips]{arximspdf}
\usepackage{mathrsfs}
\usepackage{graphicx}

\doi{10.1214/12-AOP766} 
\volume{41}
\issue{3A}
\pubyear{2013}
\firstpage{1585}
\lastpage{1627}

\makeatletter
\newcommand{\rrvert}{\vert}
\newcommand{\llvert}{\vert}

\newtheorem{theorem}{Theorem}
\newtheorem{pro}{Proposition}
\newtheorem{lem}{Lemma}
\newtheorem{corollary}{Corollary}

\newproclaim{remark}{Remark}

\newcommand{\id}{\operatorname{Id}}

\newcommand{\cb}{\operatorname{CB}}
\newcommand{\cbi}{\operatorname{CBI}}
\newcommand{\ivp}{\operatorname{IVP}}
\newcommand{\sip}{\mathbb{P}}
\newcommand{\se}{\mathbb{E}}
\newcommand{\ssa}{\mathscr{F}}
\newcommand{\si}{{\mathbf{1}}}
\newcommand{\eps}{\varepsilon}
\newcommand{\re}{\mathbb{R}}
\newcommand{\ra}{\mathbb{Q}}
\newcommand{\F}{\ssa}
\newcommand{\G}{\mathscr{G}}
\newcommand{\p}{\sip}

\makeatother

\begin{document}
\begin{frontmatter}

\title{A Lamperti-type representation of continuous-state branching
processes with immigration}
\runtitle{Lamperti transformation for CBI processes}

\begin{aug}
\author[A]{\fnms{M.~Emilia}~\snm{Caballero}\ead[label=e1]{marie@matem.unam.mx}},
\author[B]{\fnms{Jos\'e~Luis}~\snm{P\'erez~Garmendia}\ead[label=e2]{jlapg20@bath.ac.uk}}
\and
\author[A]{\fnms{Ger\'onimo}~\snm{Uribe~Bravo}\corref{}\thanksref{t2}\ead[label=e3]{geronimo@matem.unam.mx}\ead[label=u1,url]{http://www.matem.unam.mx/geronimo}}
\runauthor{M. E. Caballero, J. L. P\'erez Garmendia and G. Uribe Bravo}
\affiliation{Universidad Nacional Aut\'onoma de M\'exico, University
of Bath and
Universidad~Nacional Aut\'onoma de M\'exico}
\address[A]{M. E. Caballero\\
G. Uribe Bravo\\
Instituto de Matem\'aticas UNAM\\
\'Area de la investigaci\'on cient\'ifica\\
Circuito Exterior de Ciudad Universitaria\\
Distrito Federal CP 04510\\
M\'exico\\
\printead{e1}\\
\hphantom{E-mail: }\printead*{e3}\\
\printead{u1}}
\address[B]{J. L. P\'erez Garmendia\\
Department of Statistics\\
Instituto Tecnol\'ogico Aut\'onomo de M\'exico\\
Rio Hondo No.1, Col. Progreso Tizap\'an\\
Distrito Federal CP 01080\\
M\'exico} 
\end{aug}

\thankstext{t2}{Supported by a postdoctoral fellowship from
UC MexUS---CoNaCyt and NSF Grant DMS-08-06118.}

\received{\smonth{1} \syear{2011}}
\revised{\smonth{2} \syear{2012}}

%
\begin{abstract}
Guided by the relationship between the breadth-first walk of a rooted
tree and its sequence of generation sizes, we are able to include
immigration in the Lamperti representation of continuous-state
branching processes. We provide a representation of continuous-state
branching processes with immigration by solving a random ordinary
differential equation driven by a pair of independent L\'evy processes.
Stability of the solutions is studied and gives, in particular, limit
theorems (of a type previously studied by Grimvall, Kawazu and Watanabe
and by Li) and a simulation scheme for continuous-state branching
processes with immigration. We further apply our stability analysis to
extend Pitman's limit theorem concerning Galton--Watson processes
conditioned on total population size to more general offspring laws.
\end{abstract}

%
\begin{keyword}[class=AMS]
\kwd{60J80}
\kwd{60F17}
\end{keyword}
\begin{keyword}
\kwd{L\'evy processes}
\kwd{continuous branching processes with immigration}
\kwd{time-change}
\end{keyword}

\pdfkeywords{60J80, 60F17, Levy processes,
continuous branching processes with immigration,
time-change}

\end{frontmatter}

\section{Introduction}
\label{introSection}
\subsection{Motivation}
\label{MotivationSubsection}
In this document, we extend the Lamperti representation of continuous
state branching processes so that it allows immigration.
First, we will see how to find discrete (and simpler) counterparts to
our results in terms of the familiar Galton--Watson process with
immigration and its representation using two independent random walks.

Consider a genealogical structure with immigration such as the one
depicted in Figure~\ref{GSI}.
%
%
\begin{figure}

\includegraphics{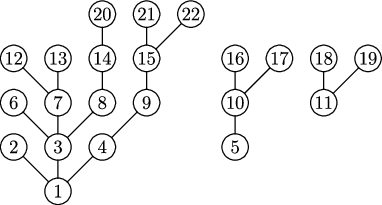}

\caption{A genealogical structure allowing immigration.}
\label{GSI}
\end{figure}
When ordering its elements in breadth-first order (with the accounting
policy of numbering immigrants after the established population in each
generation), $\chi_i$ will denote the number of children of individual~$i$. Define a first version of the breadth-first walk $\tilde
x= ( \tilde x_i )$ by
\[
\tilde x_0=0 \quad\mbox{and}\quad \tilde x_{i+1}=\tilde
x_i+\chi_{i+1}.
\]
Consider also the immigration process $y= ( y_n )_{n\geq0}$ where
$y_n$ is the quantity of immigrants arriving at generations less than
or equal to $n$ (not counting the initial members of the population as
immigrants). Finally, suppose the initial population has $k$ members.
If $c_n$ denotes the number of individuals of generations $0$ to $n$,
$c_{n+1}$ is obtained from $c_n$ by adding the quantity of sons of each
member of the $n$th generation plus the immigrants, leading to
\[
c_{n+1}=c_n+ ( \chi_{c_{n-1}+1}+\cdots+
\chi_{c_{n}} )+ ( y_{n+1}-y_n ).
\]
By induction we get
\[
c_{n+1}=k+\tilde x_{c_n}+y_{n+1}.
\]
Let $z_n$ denote the number of individuals of generation $n$ so that
$z_0=c_0=k$ and for $n\geq1$
\[
z_n=c_{n}-c_{n-1};
\]
if $\eta_i=\chi_i-1$, we can define a second version of the
breadth-first walk of the population by setting
\[
x_0=0 \quad\mbox{and}\quad x_i=x_{i-1}+
\eta_i
\]
(so that $x_i=\tilde x_i-i$). We then obtain
%
%
\begin{equation}
\label{discreteLampertiTransformation} z_{n+1}=k+x_{c_n}+y_{n+1}.
\end{equation}
This representation of the sequence of generation sizes $z$ in terms of
the breadth-first walk $x$ and the immigration function $y$ can be seen
as a \textit{discrete Lamperti transformation}.
It is the discrete form of the result we aim at analyzing. However, we
wish to consider a random genealogical structure which is not discrete.
Randomness will be captured by making the quantity of sons of
individuals an i.i.d. sequence independent of the i.i.d. sequence of
immigrants per generation, so that the model corresponds to a
Galton--Watson with immigration. Hence $x$ and $y$ would become two
independent random walks, whose jumps take values in $ \{
-1,0,1,\ldots\}$ and $ \{ 0,1,\ldots\}$,
respectively. Discussion of
nondiscreteness in the random genealogy model would take us far apart
[we are motivated by L\'evy trees with or without immigration,
discussed, e.g., by \citet
{MR1954248,MR1883717,MR2485021,MR2500236}]. We only mention that
continuum trees are usually defined through a continuum analogue of the
depth-first walk; our point of view is that generation sizes should be
obtained in terms of the continuum analogue of the breadth-first walk.
Indeed, in analogy with the discrete model, we just take $X$ and $Y$ as
independent L\'evy processes, the former without negative jumps (a
spectrally positive L\'evy process) and the latter with increasing
sample paths (a subordinator). The discrete Lamperti transformation of
(\ref{discreteLampertiTransformation}) then takes the form
%
%
\begin{equation}
\label{generalizedLampertiTransformation} Z_t=x+X_{\int_0^t Z_s \,ds}+Y_t.
\end{equation}
This should be the continuum version of a Galton--Watson process with
immigration, namely, the continuous-state branching processes with
immigration introduced by \citet{MR0290475}.
\subsection{Preliminaries}
\subsubsection{(Possibly killed) L\'evy processes}
A spectrally positive L\'evy process (spLp) is a stochastic process
$X= ( X_t )_{t\geq0}$ which starts at zero, takes values on
$(-\infty,\infty]$, has independent and stationary increments,
c\`adl\`ag paths, and no negative jumps. Such a process is characterized
by its Laplace exponent $\Psi$ by means of the formula
\[
\se\bigl( e^{-\lambda X_t} \bigr)=e^{t{\Psi( \lambda)}},
\]
where
\[
{\Psi( \lambda)}=-\kappa+a \lambda+\frac{\sigma^2\lambda^2}{2}+\int
_0^\infty
\bigl( e^{-\lambda x}-1+\lambda x\si_{x\leq
1} \bigr) {\nu( dx )}
\]
for $\lambda>0$; here $\nu$ is the so-called L\'evy measure on
$(0,\infty)$ and satisfies
\[
\int1\wedge x^2 {\nu( dx )}<\infty.
\]
The constant $\kappa$ will be for us the killing rate; a L\'evy
process with killing rate $\kappa$ can be obtained from one with zero
killing rate by sending the latter to $\infty$ at an independent
exponential time of parameter $\kappa$; $\sigma^2$ is called the
diffusion coefficient, while $a$ is the drift. 

We shall also make use of subordinators, which are spLp with increasing
trajectories. The Laplace exponent $\Phi$ of a subordinator $X$ is
defined as the negative of its Laplace exponent as a spLp, so
\[
\se\bigl( e^{-\lambda X_t} \bigr)=e^{-t{\Phi( \lambda
)}}.
\]
Since the L\'evy measure $\nu$ of a subordinator actually satisfies
\[
\int1\wedge x {\nu( dx )}<\infty,
\]
and subordinators have no Brownian component (i.e., $\sigma^2=0$), we
can write
\[
{\Phi( \lambda)}=\kappa+d\lambda+\int\bigl( 1-e^{-\lambda
x} \bigr) {\nu(
dx )}.
\]
So, we have the relationship
\[
-d=a+\int_0^1 x {\nu( dx )}
\]
between the parameters of $X$ seen as a spLp and as a subordinator.

\subsubsection{Continuous-state branching processes and the Lamperti
representation}
Continuous-state branching ($\cb$) processes are the continuous
time and space version of Galton--Watson processes. They were
introduced in different levels of generality by \citet{1958Jirina},
\citet{1967Lamperti} and \citet{1967Silverstein}. They are Feller
processes with state-space $[0,\infty]$ (with any metric that makes it
homeomorphic to $[0,1]$) satisfying the following branching property:
the sum of two independent copies started at $x$ and $y$ has the law of
the process started at $x+y$. The states $0$ and $\infty$ are
absorbing. The branching property can be recast by stating that the
logarithm of the Laplace transform of the transition semigroup is given
by a linear transformation of the initial state.

As shown by \citet{1967Silverstein}, $\cb$ processes are in one to one
correspondence with Laplace exponents of (killed) spectrally positive
L\'evy processes, which are called the branching mechanisms. In short,
the logarithmic derivative of the semigroup of a $\cb$ process at zero
applied to the function $x\mapsto e^{-\lambda x}$ exists and is equal
to $x\mapsto x{\Psi( \lambda)}$. The function $\Psi$
is the
called the \textit{branching mechanism} of the $\cb$ process and it is
the Laplace exponent of a spLp. A probabilistic form of this assertion
is given by \citet{1967LampertiCSBP} who states that if $X$ is a spLp
with Laplace exponent $\Psi$, and for $x\geq0$, we set $T$ for its
hitting time of $-x$,
\[
I_t=\int_0^t\frac{1}{x+X_{s\wedge T}} \,ds
\]
and $C$ equal to its right-continuous inverse, then
\[
Z_t=x+X_{C_{t\wedge T}}
\]
is a $\cb$ process with branching mechanism $\Psi$, or
${\cb( \Psi)}$. This does not seem to be directly
related to
(\ref{generalizedLampertiTransformation}). The fact that it is related
gives us what we think is the right perspective on the Lamperti
transformation and the generalization considered in this work. Indeed,\vadjust{\goodbreak}
as previously shown in Ethier and Kurtz [(\citeyear{MR838085}),
Chapter 6, Section 1], $Z$
is the only process satisfying
%
%
\begin{equation}
\label{ODEFormOfLampertiTransformation}
Z_t=x+X_{\int_0^t Z_s
\,ds},
\end{equation}
which is absorbed at zero. This is (\ref
{generalizedLampertiTransformation}) in the absence of immigration. To
see that a process satisfying (\ref{ODEFormOfLampertiTransformation})
can be obtained as the Lamperti transform of~$X$, note that if
$C_t=\int_0^t Z_s \,ds$, then while $Z$ has not reached zero, $C$ is
strictly increasing so that it has an inverse, say $I$, whose
right-hand derivative $I'_+$ is given by
\[
{I'_{+} ( t )}=\frac{1}{{C'_{+} ( I_t
)}}=\frac{1}{Z_{I_t}}=
\frac
{1}{x+X_{{C\circ I ( t )}}}=\frac{1}{x+X_t}.
\]

\subsubsection{Continuous-state branching processes with immigration}
Contin\-uous-state branching processes with immigration (or $\cbi$
processes) are the continuous time and space version of Galton--Watson
processes with immigration and were introduced by
\citet{MR0290475}. They are Feller processes with state-space
$[0,\infty]$ such that the logarithm of the Laplace of the transition
semigroup is given by an affine transformation of the initial state.
[They thus form part of the affine processes studied by
\citet{MR2243880}.] As shown by \citet{MR0290475}, they are
characterized by the Laplace exponents of a spLp and of a subordinator:
the logarithmic derivative of the semigroup of a $\cb$ process at zero
applied to the function $x\mapsto e^{-\lambda x}$ exists and is equal
to the function
\[
x\mapsto x{\Psi( \lambda)}-{\Phi( \lambda)},
\]
where $\Psi$ is the Laplace exponent of a spLp and $\Phi$ is the
Laplace exponent of a subordinator. They are, respectively, called the
\textit{branching and immigration mechanisms} and characterize the
process which is therefore named ${\cbi( \Psi,\Phi)}$.

We aim at a probabilistic representation of $\cbi$ processes in the
spirit of the Lamperti representation.

\subsection{Statement of the results}

We propose to construct a ${\cbi( \Psi,\Phi)}$ that
starts at
$x$ by solving the functional equation
%
%
\begin{equation}
\label{genLampRepDef} Z_t=x+X_{\int_0^t Z_s \,ds}+Y_t.
\end{equation}
We call such a process $Z$ the \textit{Lamperti transform} of $ (
X,x+Y )$ and denote it by $Z={L ( X,x+Y )}$; however,
the first thing
to do is to show that there exists a unique process which satisfies
(\ref{genLampRepDef}). When $Y$ is zero, a particular solution to
(\ref{genLampRepDef}) is the Lamperti transform of $X+x$ recalled
above. Even in this case there could be many solutions to (\ref
{genLampRepDef}), in clear contrast to the discrete case where one can
proceed recursively to construct the unique solution.
Our stepping stone for the general analysis of (\ref{genLampRepDef})
is the following partial result concerning existence and uniqueness
proved in Section~\ref{ODESection}.\vadjust{\goodbreak}

A pair of c\`adl\`ag functions $ ( f,g )$ such that $f$ has no
negative jumps, $g$ is nondecreasing and ${f ( 0 )}+{g
( 0 )}\geq
0$ is termed an \textit{admissible breadth-first pair}; $f$ and $g$
will be termed the \textit{reproduction and immigration functions},
respectively. When $g$ is constant, we say that $f+g$ is
\textit{absorbed at zero} if ${f ( x )}+g=0$ implies ${f (
y )}+g=0$ for
all $y>x$.
%
%
\begin{theorem}
\label{ExistenceTheorem}
Let $ ( f,g )$ be an admissible breadth-first pair. There
exists a
nonnegative $h$ satisfying the equation
\[
{h ( t )}={f \biggl( \int_0^t {h ( s )} \,ds
\biggr)}+{g ( t )}.
\]
Furthermore, the solution is unique when $g$ is strictly increasing,
when $f+{g ( 0 )}$ is a strictly positive function or when
$g$ is
constant and $f+g$ is absorbed at zero.
\end{theorem}

In the context of Theorem~\ref{ExistenceTheorem}, much is gained by
introducing the function $c$ given by
\[
{c ( t )}=\int_0^t {h ( s )} \,ds,
\]
which has a right-hand derivative $c'_+$ equal to $h$. This is because
the functional equation for $h$ can then be recast as the initial value problem
\[
{\ivp( f,g )}= \cases{ c'_+=f\circ c+g,
\cr
{c ( 0 )}=0.}
\]

Our forthcoming approximation results for the function $h$ of Theorem~\ref{ExistenceTheorem} rely on the study of a functional inequality.
Let $ ( f,g )$ be an admissible breadth-first pair. We will be
interested in functions $c$ which satisfy
%
%
\begin{eqnarray}
\label{IVPWithInequalities}
&&\int_s^t {f_-
\circ c ( r )}+{g ( r )} \,dr
\leq{c ( t )}-{c ( s )}\leq\int_s^t
{f\circ c ( r )}+{g ( r )} \,dr\nonumber\\[-8pt]\\[-8pt]
&&\eqntext{\mbox{for $s\leq t$.}}
\end{eqnarray}
Note that any solution $c$ to ${\ivp( f,g )}$ satisfies
(\ref
{IVPWithInequalities}): the second inequality is actually an equality by
definition of ${\ivp( f,g )}$, and since $f\geq f_-$ as
$f$ has no
negative jumps, we get the first inequality. Hence, the functional
inequality (\ref{IVPWithInequalities}) admits solutions. Regarding
uniqueness, if the solution to (\ref{IVPWithInequalities}) is unique,
then the solution to ${\ivp( f,g )}$ is unique, and since the
latter is nonnegative and nondecreasing, so is the former.
Also, similar sufficient conditions for uniqueness of ${\ivp(
f,g )}$ of Theorem~\ref{ExistenceTheorem} imply uniqueness of
nondecreasing solutions of the functional inequality~(\ref
{IVPWithInequalities}).
%
%
\begin{pro}
\label{UniquenessForIVPWithInequalitiesProposition}
Let $ ( f,g )$ be an admissible breadth-first pair. If
either $g$
is strictly increasing, $f_-+{g ( 0 )}$ is strictly
positive or $g$
is constant and $f_-+{g ( 0 )}$ is absorbed at zero, then
(\ref
{IVPWithInequalities}) has an unique nondecreasing solution starting
at zero.\vadjust{\goodbreak}
\end{pro}
However, as is shown in Section
\ref{ProofOfAnalyticAssertionsSubsection}, assuming that (\ref
{IVPWithInequalities}) admits an unique solution is stronger than just
assuming that ${\ivp( f,g )}$ has an unique solution.

As a consequence of the analytic Theorem~\ref{ExistenceTheorem}, we
solve a probabilistic question raised by
Lambert (\citeyear{LambertQProcess,MR2299923}).
%
%
\begin{corollary}
\label{LambertProblemCorollary}
Let $X$ be a spectrally positive $\alpha$-stable L\'evy process. For
any c\`adl\`ag and strictly increasing process $Y$ independent of $X$,
there is weak existence and uniqueness for the stochastic differential equation
%
%
\begin{equation}
\label{SDEStablespLp} Z_t=x+\int_0^t
\llvert Z_s\rrvert^{1/\alpha} \,dX_s+Y_t.
\end{equation}
\end{corollary}
When $X$ is twice a Brownian motion and $Y_t=\delta t$ for some $\delta
>0$, this might be one of the simplest proofs available of weak
existence and uniqueness of the SDE defining squared Bessel processes,
since it makes no mention of the Tanaka formula or local times; it is
based on Knight's theorem and Theorem~\ref{ExistenceTheorem}. When $X$
is a Brownian motion and $dY_t={b ( t )} \,dt$ for some
Lipschitz and
deterministic ${b\dvtx[0,\infty)\to[0,\infty)}$, \citet{MR770393}
actually proves pathwise uniqueness through a local time argument. Our
result further shows that if $b$ is measurable and strictly positive,
then there is weak uniqueness.
In the case $Y$ is an $ ( \alpha-1 )$-stable subordinator
independent of $X$, we quote Lambert
(\citeyear{LambertQProcess,MR2299923}):

\begin{quote}
\textit{\ldots whether or not uniqueness holds for} (\ref{SDEStablespLp})
\textit{remains an open question.}
\end{quote}

\noindent Corollary~\ref{LambertProblemCorollary} answers affirmatively. Note
that when $Y=0$, the stated result follows from \citet{MR1905857},
and is handled by a time-change akin to the Lamperti transformation.
\citet{MR2584896} obtain strong existence and pathwise uniqueness
for a different kind of SDE related to $\cbi$ processes with stable
reproduction and immigration.

Regarding solutions to (\ref{genLampRepDef}), Theorem \ref
{ExistenceTheorem} is enough to obtain the process $Z$ when the
subordinator $Y$ is strictly increasing. When $Y$ is compound Poisson,
a solution to (\ref{genLampRepDef}) can be obtained by pasting
together Lamperti transforms.
However, further analysis using the pathwise behavior of $X$ when $Y$
is zero or compound Poisson implies the following result.
%
%
\begin{pro}
\label{uniquenessForLevy}
Let $x\geq0$, $X$ be a spectrally positive L\'evy process and $Y$ an
independent subordinator. Then there is a unique c\`adl\`ag process $Z$
which satisfies
\[
Z_t=x+X_{\int_0^t Z_s \,ds}+Y_t.
\]
The above equation is satisfied by any c\`adl\`ag process $Z$ satisfying
the functional inequality
\[
x+X_{\int_0^t Z_s \,ds-}+Y_{t}\leq Z_t\leq
x+X_{\int_0^t Z_s \,ds}+Y_t,
\]
which also has a unique solution.
\end{pro}

Our main result, a pathwise construction of a ${\cbi( \Psi
,\Phi)}$, is the following.

%
\begin{theorem}
\label{CBIRepThm}
Let $X$ be a spectrally positive L\'evy process with Laplace exponent
$\Psi$ and $Y$ an independent subordinator with Laplace exponent $\Phi
$. The unique stochastic process $Z$ which solves
\[
Z_t=x+X_{\int_0^t Z_s \,ds}+Y_t
\]
is a ${\cbi( \Psi,\Phi)}$ that starts at $x$.
\end{theorem}

We view Theorems~\ref{ExistenceTheorem} and~\ref{CBIRepThm} as a
first step in the construction of branching processes with immigration
where the immigration can depend on the current value of the
population. One generalization would be to consider solutions to
\[
Z_t=x+X_{\int_0^t {a ( s,Z_s )} \,ds}+Y_{\int_0^t {b
( s,Z_s )} \,ds},
\]
where $a$ is interpreted as the breeding rate, and $b$ as the rate at
which the arriving immigration is incorporated into the population. For
example, \citet{MR2500236} consider a continuous branching process
where immigration is proportional to the current state of the
population. This could be modeled by the equation
\[
Z_t=x+X_{\int_0^t Z_s \,ds}+Y_{\int_0^t \alpha Z_s \,ds},
\]
which, thanks to the particular case of Theorem~\ref{CBIRepThm} stated
by \citet{1967LampertiCSBP}, has the law of a ${\cb( \Psi
-\alpha\Phi)}$ started at $x$; this is the conclusion of
\citet{MR2500236},
where they rigorously define the model in terms of a Poissonian
construction of a more general class of CBI processes which is inspired
in previous work of \citet{MR656509} for CBIs with continuous sample
paths. Another representation of CBI processes, this time in terms of
solutions to stochastic differential equations was given by
\citet{MR2243880} under moment conditions.

The usefulness of Theorem~\ref{CBIRepThm} is two-fold: first, we can
use known sample path properties of $X$ and $Y$ to deduce sample-path
properties of $Z$, and second, this representation gives a particular
coupling with monotonicity properties which are useful in limit
theorems involving $Z$, as seen in
Corollaries~\ref{ContinuityOfCBILawsCorollary},
\ref{LimitTheoremGWI} and Theorem \ref
{ExtensionOfPitmansTheorem}. Simple applications of Theorem \ref
{CBIRepThm} include the following.
%
%
\begin{corollary}[{[\citet{MR0290475}]}]
\label{KawazuWatanabeCorollary}
If $\Psi$ is the Laplace exponent of a spectrally positive L\'evy
process, and $\Phi$ is the Laplace exponent of a subordinator, there
exists a CBI process with branching mechanism $\Psi$ and
immigration mechanism $\Phi$.
\end{corollary}
%
%
\begin{corollary}
\label{NoDownwardJumpsCorollary}
A ${\cbi( \Psi,\Phi)}$ process does not jump downward.\vadjust{\goodbreak}
\end{corollary}
\citet{MR2592395} give a direct proof of this when $\Phi=0$.
%
%
\begin{corollary}
\label{LIL}
Let $Z$ be a ${\cbi( \Psi,\Phi)}$ that starts at
$x>0$, let
$\tilde\Phi$ be the right-continuous inverse of $\Psi$, and define
\[
{\alpha( t )}=\frac{{\log}\llvert{\log t}\rrvert}{{\tilde
\Phi( {t^{-1}\log}\llvert{\log t}\rrvert)}}.
\]
There exists a constant $\zeta$ (in general nonzero) such that
\[
\liminf_{t\to0}\frac{Z_t-x}{{\alpha( xt )}}=\zeta.
\]
\end{corollary}
The case $x=0$ in Corollary~\ref{LIL} is probably very different, as
seen when
${\Psi( \lambda)}=2\lambda^2$ and ${\Phi(
\lambda)}=d\lambda$,
which corresponds to the squared Bessel process of dimension $d$.
Indeed, It{\^o} and McKean [(\citeyear{MR0345224}), page 80] show that for a squared Bessel
process $Z$ of integer dimension that starts at $0$, we have
\[
\limsup_{t\to0 }\frac{Z_t}{ {2t\log}\llvert{\log t}\rrvert}=1.
\]
We have not been able to obtain this result using the Lamperti
transformation. However, note that starting from positive states, we
can obtain the lower growth rate, since it is the reproduction function
$X$ that determines it, while starting from~$0$, it is probably a
combination of the local growth of $X$ and $Y$ that drives that of $Z$.

A solution $c$ to ${\ivp( f,g )}$ is said to \textit
{explode} if
there exists $t\in(0,\infty)$ such that ${c ( t )}=\infty$.
(Demographic) explosion is an unavoidable phenomena of ${\ivp(
f,g )}$. When $f>0$ and $g=0$, it is known that explosion occurs
if and
only if
\[
\int^\infty\frac{1}{{f ( x )}} \,dx<\infty.
\]
Actually, even when there is immigration, the main function responsible
for explosion is the reproduction function.
%
%
\begin{pro}
\label{DeterministicExplosionProposition}
Let $ ( f,g )$ be an admissible pair, and let $f^+={\max
( f,0 )}$.

\begin{longlist}[(2)]
\item[(1)]\label{DeterministicExplosionProposition1} If
$\int^{\infty}1/{f^+ ( x )} \,dx=\infty$,
then no solution to ${\ivp( f,g )}$ explodes.
\item[(2)]\label{DeterministicExplosionProposition2} If
$\int^{\infty}1/{f^+ ( x )} \,dx<\infty$, $\lim_{x\to
\infty
}{f ( x )}=\infty$ and ${g ( \infty)}$
exceeds the maximum of
$-f$, then any solution to ${\ivp( f,g )}$ explodes.
\end{longlist}
\end{pro}
We call $f$ an \textit{explosive reproduction function} if
\[
\int^{\infty}\frac{1}{{f^+ ( x )}} \,dx<\infty.
\]

Recall that $\infty$ is an absorbing state for CBI processes;
Proposition~\ref{DeterministicExplosionProposition} has immediate
implications on how\vadjust{\goodbreak} a CBI process might reach it. First of all, CBI
processes might jump to $\infty$, which happens if and only if either
the branching or the immigration corresponds to killed L\'evy
processes. When there is no immigration and the branching mechanism
$\Psi$ has no killing rate, the criterion is due to
\citet{1970Ogura}
and \citet{MR0408016}, who assert that the probability that a ${\cb
( \Psi)}$ started from $x>0$ is absorbed at infinity in finite
time is positive if and only if
\[
\int_{0+}\frac{1}{{\Psi( \lambda)}} \,d\lambda>-\infty.
\]
One can even obtain a formula for the distribution of its explosion
time; cf. the proof of Theorem 2.2.3.2 in
\citet{MR2466449}, page 95.
We call such $\Psi$ an \textit{explosive branching mechanism}. From
Proposition~\ref{DeterministicExplosionProposition} and Theorem \ref
{CBIRepThm} we get:

%
\begin{corollary}
\label{CBIExplosionCorollary}
Let $x>0$.
\begin{longlist}[(3)]
\item[(1)] The probability that a ${\cbi( \Psi,\Phi
)}$ $Z$ that
starts at $x$ jumps to $\infty$ is positive if and only if ${\Psi
( 0 )}$ or ${\Phi( 0 )}$ are nonzero.
\item[(2)] The probability that $Z$ reaches $\infty$ continuously is
positive if and only if ${\Psi( 0 )}=0$ and $\Psi$ is an
explosive
branching mechanism.
\item[(3)] The probability that $Z$ reaches $\infty$ continuously is
equal to $1$ if ${\Psi( 0 )}={\Phi( 0
)}=0$, $\Phi$ is not
zero and $\Psi$ is explosive.
\end{longlist}
\end{corollary}
%

We mainly use stochastic integration by parts in our proof of Theorem~\ref{CBIRepThm}; however, a weak convergence type of proof, following
the case $\Phi=0$ presented by \citet{MR2592395}, could also be
achieved in conjunction with a stability result, based on the
forthcoming Theorem~\ref{StabilityTheorem}.

The following result deals with stability of ${\ivp( f,g
)}$ under
changes in $f$ and $g$ and even includes a discretization of the
initial value problem, itself. Indeed, consider the following
approximation procedure: given $\sigma>0$, called the \textit{span},
consider the partition
\[
t_i=i\sigma,\qquad i=0,1,2,\ldots,
\]
and construct a function $c^\sigma$ by the recursion
\[
{c^\sigma( 0 )}=0
\]
and for $t\in[t_{i-1},t_i)$,
\[
{c^\sigma( t )}={c^\sigma( t_{i-1} )}+ (
t-t_{i-1} ) \bigl[ {f\circ c^\sigma( t_{i-1} )}+{g (
t_{i-1} )} \bigr]^+.
\]
Equivalently, the function $c^\sigma$ is the unique solution to the equation
\[
{\ivp_\sigma( f,g )}: 
{c^\sigma( t )}=\int
_0^t \bigl[ {f\circ c^{\sigma} \bigl(
\lfloor s/\sigma\rfloor\sigma\bigr)}+{g \bigl( \lfloor s/\sigma\rfloor
\sigma
\bigr)} \bigr]^+ \,ds. 
\]
We will write ${\ivp_0 ( f,g )}$ to mean ${\ivp(
f,g )}$. Let
$D_+$ denote the right-hand derivative.

The stability result is stated in terms of the usual \textit{Skorohod
$J_1$ topology} for c\`adl\`ag functions: a sequence $f_n$ converges to
$f$\vadjust{\goodbreak} if there exist a sequence of homeomorphisms of $[0,\infty)$ into
itself such that
\[
f_n-f\circ\lambda_n \quad\mbox{and}\quad \lambda_n-\id
\qquad\mbox{converge to zero uniformly on compact sets}
\]
(where $\id$ denotes the identity function on $[0,\infty)$). However,
part of the theorem uses another topology on nonnegative c\`adl\`ag
functions introduced by \citet{MR2592395}, which we propose to
call the
\textit{uniform $J_1$ topology}. Consider a distance $d$ on $[0,\infty
]$ which makes it homeomorphic to $[0,1]$. Then the uniform $J_1$
topology is characterized by the following: a sequence $f_n$ converges
to $f$ if there exist a sequence of homeomorphisms of $[0,\infty)$
into itself such that
\[
{d ( f_n,f\circ\lambda_n )}\to0 \quad\mbox{and}\quad
\lambda_n-\id\to0\qquad\mbox{uniformly on $[0,\infty)$}.
\]

%
%
\begin{theorem}
\label{StabilityTheorem}
Let $ ( f,g )$ be an admissible breadth-first pair
and suppose there is a unique nondecreasing function $c$ which
satisfies ${c ( 0 )}=0$ and (\ref{IVPWithInequalities})
[and is
therefore the unique solution to ${\ivp( f,g )}$]; define its
explosion time by
\[
\tau=\inf\bigl\{ t\geq0\dvtx{c ( t )}=\infty\bigr\}\in(0,\infty].
\]

Let $ ( f_n,g_n )$ be admissible breadth-first pairs. Suppose
$f_n\to f$ and $g_n\to g$ in the Skorohod $J_1$ topology and that
$\sigma_n$ is a sequence of nonnegarive real numbers which tend to
zero. Let $c_n$ be the unique solution to ${\ivp_{\sigma_n} (
f_n,g_n )}$ when $\sigma_n>0$ and any solution to ${\ivp
( f_n,g_n )}$ when $\sigma_n=0$. Then $c_n\to c$ pointwise and uniformly
on compact sets of $[0,\tau)$.

Furthermore, if $f\circ c$ and $g$ do not jump at the same time, then
\mbox{$D_+ c_n\to D_+ c$}:

\begin{longlist}[(2)]
\item[(1)] in the Skorohod $J_1$ topology if $\tau=\infty$, and
\item[(2)] in the uniform $J_1$ topology if $\tau<\infty$ if we
additionally assume that ${f_n ( x )},{f ( x
)}\to\infty$ as
$x\to\infty$ uniformly in $n$.
\end{longlist}
\end{theorem}

It is not very hard to show that the jumping condition of Theorem \ref
{StabilityTheorem} holds in a stochastic setting.
%
%
\begin{pro}
\label{JumpingConditionProposition}
Let $X$ be a spLp, $Y$ an independent subordinator with Laplace
exponents $\Psi$ and $\Phi$ and, for $x\geq0$, let $Z$ the unique
process such that
\[
Z_t=x+X_{C_t}+Y_t\qquad\mbox{where
}C_t=\int_0^t Z_s \,ds.
\]
Almost surely, the processes $X\circ C$ and $Y$ do not jump at the same time.
\end{pro}

From Theorem~\ref{StabilityTheorem} and Propositions \ref
{uniquenessForLevy} and~\ref{JumpingConditionProposition}, we deduce
the following weak continuity result.
%
%
\begin{corollary}
\label{ContinuityOfCBILawsCorollary}
Let $\Psi_n,\Psi$ be Laplace exponents of spLps and $\Phi_n,\Phi$
be Laplace exponents of subordinators and suppose that $\Psi_n\to\Psi
$ and $\Phi_n\to\Phi$ pointwise. If $ ( x_n )$ is a
sequence in
$[0,\infty]$ converging to\vadjust{\goodbreak} $x$ and $Z_n$ (resp., $Z$) are CBIs with
branching and immigration mechanisms $\Psi_n$ and $\Phi_n$ (resp.,
$\Psi$ and $\Phi$) and starting at $x_n$ (resp., $x$) then $Z_n\to Z$
in the Skorohod $J_1$ topology on c\`adl\`ag paths on $[0,\infty]$ if
$\Psi$ is nonexplosive and in the uniform $J_1$ topology if $\Psi$
is explosive.
\end{corollary}

Theorem~\ref{StabilityTheorem} also allows us to simulate CBI
processes. Indeed, if we can simulate random variables with
distribution $X_t$ and $Y_t$ for every $t>0$, we can then approximately
simulate the process $Z$ as the right-hand derivative of the solution
to ${\ivp_\sigma( X,x+Y )}$. (Alternatively, if we can
approximate
$X$ and~$Y$, e.g., by compound Poisson processes with drift, we can
also apply $\ivp_\sigma$ to approximate the paths of $Z$.) The
procedure ${\ivp_\sigma( X,x+Y )}$ actually corresponds
to an
Euler method of span $\sigma$ to solve ${\ivp( X,x+Y
)}$. Theorem~\ref{StabilityTheorem} implies the convergence of the Euler method as
the span goes to zero when applied to ${\ivp( X,x+Y )}$,
even with
the discontinuous driving functions $X$ and $Y$!

We also give an application of Theorem~\ref{StabilityTheorem} to
limits of Galton--Watson processes with immigration. Let $X^n$ and
$Y^n$ be independent random walks with step distributions $\mu_n$ and
$\nu_n$ supported on $ \{ -1,0,1,\ldots\}$ and $ \{
0,1,2,\ldots\}$, and for any $k_n\geq0$, define recursively
the sequences $C^n$ and
$Z^n$ by setting
\[
C^n_0= Z^n_0=k_n,\qquad
Z^n_{m+1}=k_n+X^n_{C^n_{m}}+Y^n_{m+1}\quad
\mbox{and}\quad C^n_{m+1}=C^n_m+Z^n_{m+1}.
\]
As discussed in Section~\ref{MotivationSubsection}, the sequence
$Z^n$ is a Galton--Watson process with immigration with offspring and
immigration distributions $\mu_n$ and $\nu_n$. However, if $X^n$ and
$Y^n$ are extended by constancy on $[m,m+1)$ for $m\geq0$ (keeping the
same notation), then $C^n$ is the approximation of the Lamperti
transformation with span~$1$ applied to $X^n$ and $Y^n$ and $Z^n$ is
the right-hand derivative of $C^n$. In order to apply Theorem \ref
{StabilityTheorem} to these processes, define the scaling operators
$S_{a}^b$ by
\[
{S_{a}^bf ( t )}=\frac{1}{b}{f ( at )}.
\]

%
\begin{corollary}
\label{LimitTheoremGWI}
Suppose the existence of sequences $a_n,b_n$ such that
\[
X^n_{a_n}/n \quad\mbox{and}\quad Y^n_{b_n}/n
\]
converge weakly to the infinitely divisible distributions $\mu$ and
$\nu$ corresponding to a spectrally positive L\'evy process and a
subordinator; denote by $\Psi$ and $\Phi$ their Laplace exponents.
Suppose that $b_n\to\infty$ and, for any $\alpha>0$, $a_{\lfloor
\alpha n\rfloor}/n\to\infty$. Let $k_n\to\infty$, and suppose that either
\[
\frac{k_n b_{\lfloor k_n/x\rfloor}}{x a_{\lfloor k_n/x\rfloor}}\to c\in
[0,\infty) \quad\mbox{or}\quad \frac{x a_{\lfloor k_n/x\rfloor}}{k_n
b_{\lfloor k_n/x\rfloor}}\to0
\]
as $n\to\infty$. Setting $e_n=b_{\lfloor k_n/x\rfloor}$ in the first case
and $e_n=xa_{\lfloor k_n/x\rfloor}/k_n$ in the second, we have that
\[
S^{k_n/x}_{e_n}Z^n\vadjust{\goodbreak} 
\]
converges in distribution, toward a ${\cbi( c\Psi,\Phi
)}$ in the
first case and toward a ${\cb( \Psi)}$ in the second. The
convergence takes place in the Skorohod $J_1$ topology if $\Psi$ is
nonexplosive and in the uniform $J_1$ topology, otherwise.
\end{corollary}
%
When $\Psi$ is nonexplosive and $\Phi=0$, the above theorem was
proved by \citet{MR0362529}. He also proved the convergence of
finite-dimensional distributions in the explosive case, which we
complement with a limit theorem. For general $\Phi$, but nonexplosive
$\Psi$, a similar result was proven by \citet{MR2225068}. However, as
will be seen in the proof (which relies on the stability of the
Lamperti transformation stated in Theorem~\ref{StabilityTheorem}), if
the convergence of $S^n_{a_n}X^n$ and $S^n_{b_n}Y^n$ takes place almost
surely, then $S^n_{e_n}Z^n$ also converges almost surely.


The stability result of Theorem~\ref{StabilityTheorem} applies not
only in the Markovian case of CBI processes. As an example, we
generalize work of \citet{MR1681110} who considers the scaling limits
of conditioned Galton--Watson processes in the case of the Poisson
offspring distribution. Let $\mu$ be an offspring distribution with
mean $1$ and suppose that $Z^{n}$ is a Galton--Watson process started
at $k_n$ and conditioned on
\[
\sum_{i=0}^\infty Z^{n}_i=n.
\]
We shall consider the scaling limit of $Z^{n}$ as $n\to\infty$
whenever the shifted reproduction law $\tilde\mu_k=\mu_{k+1}$ is in
the domain of attraction of a stable law without the need of centering.
The scaling limit of a random walk with step distribution $\tilde\mu$
is then a spectrally positive stable law of index $\alpha\in(1,2]$
with which one can define, for every $l>0$ the first passage bridge
$F^{l}$ starting at $l$ and ending at $0$ of length $1$ of the
associated L\'evy process. Informally this is the stable process
started at $l$, conditioned to be above $0$ on $[0,1]$ and conditioned
to end at $0$ at time $1$. This intuitive notion was formalized by
\citet{MR2534486}. The Lamperti transform of $F^l$ will be the
right-hand derivative of the unique solution to ${\ivp(
F^l,0 )}$.
%
%
\begin{theorem}
\label{ExtensionOfPitmansTheorem}
Let $Z^{n}$ be a Galton--Watson process with critical offspring law
$\mu$ which starts at $k_n$ and is conditioned on $\sum_{i=1}^\infty
Z^{n}_i=n$. Let $S$ be a random walk with step distribution $\mu$ and
suppose there exist constants $a_n\to\infty$ such that $ (
S_n-n )/a_n$ converges in law to a spectrally positive stable
distribution with Laplace exponent $\Psi$. Let $X$ be a L\'evy process
with Laplace exponent $\Psi$ and $F^l$ its first passage bridge from
$l>0$ to $0$ of length $1$. If $k_n/a_n\to l$, then the sequence
\[
S_{n/a_n}^{a_n}Z^{n} 
\]
converges in law to the Lamperti transform of $F^{l}$ in the Skorohod
$J_1$ topology.
\end{theorem}
When $\alpha=2$, the process $F^l$ is a Bessel bridge of dimension $3$
between $l$ and $0$ of length $1$, up to a normalization factor. In
this case, Pitman [(\citeyear{MR1681110}), Lemma 14] tells us that the Lamperti
transform $Z^l$ of $F^l$ satisfies the SDE
\[
\cases{ dZ^l_v=2\sqrt{Z^l_v}
\,dB_v+ \biggl[ 4-\dfrac{ ( Z^l_v
)^2}{1-\int_0^v Z^l_u \,du} \biggr] \,dv,
\vspace*{2pt}\cr
Z^l_0=l,}
\]
driven by a Brownian motion $B$, and it is through stability theory for
SDEs that \citet{MR1681110} obtains Theorem \ref
{ExtensionOfPitmansTheorem} when $\mu$ is a Poisson distribution with
mean $1$.
Theorem~\ref{ExtensionOfPitmansTheorem} is a complement to the
convergence of Galton--Watson forests conditioned on their total size
and number of trees given in
\citet{MR2534486}.
When $l=0$, our techniques cease to work. Indeed, the corresponding
process $F^0$ would be a normalized Brownian excursion above zero, and
the problem ${\ivp( F^0,0 )}$ does not have a unique
solution, as
discussed at the beginning of Section~\ref{ODESection}. Hence, even if
our techniques yield tightness in the corresponding limit theorem with
$l=0$, we would have to give further arguments to prove that any
subsequential limit is the correct solution $\ivp( F^0,0
)$. The
limit theorem when $l=0$ and $\alpha=2$ was conjectured by \citet
{MR1166406}, and proved by \citet{MR1608230} by analytic methods. For
any $\alpha\in(1,2]$, the corresponding statement was stated and
proved by \citet{kersting1998height} by working with the usual Lamperti
transformation, which chooses a particular solution to ${\ivp(
F^0,0 )}$.

The paper is organized as follows. Theorem~\ref{ExistenceTheorem},
Proposition~\ref{uniquenessForLevy} and Corollary \ref
{LambertProblemCorollary} are proved in Section~\ref{ODESection} which
focuses on the analytic aspects of the Lamperti transformation and its
basic probabilistic implications. The representation CBI processes of
Theorem~\ref{CBIRepThm} is then proved in Section \ref
{CBIRepSection}, together with Proposition \ref
{JumpingConditionProposition}, Corollaries~\ref{LIL} and \ref
{CBIExplosionCorollary}. Finally, Section~\ref{StabilitySection} is
devoted to the stability of the Lamperti transformation with a proof of
Theorem~\ref{StabilityTheorem}, Proposition \ref
{UniquenessForIVPWithInequalitiesProposition},
Corollaries~\ref{ContinuityOfCBILawsCorollary},
\ref{LimitTheoremGWI} and Theorem \ref
{ExtensionOfPitmansTheorem}. (Corollaries \ref
{KawazuWatanabeCorollary} and~\ref{NoDownwardJumpsCorollary} are
considered to follow immediately from Theorem~\ref{CBIRepThm}; proofs
have been omitted.)

\section{The generalized Lamperti
transformation as an initial value problem}
\label{ODESection}
Let $ ( f,g )$ be an admissible breadth-first pair, meaning that
$f$ and $g$ are c\`adl\`ag functions with $g$ increasing, $f$ without
negative jumps and ${f ( 0 )}+{g ( 0 )}\geq0$.
We begin by
studying the existence of a nonnegative c\`adl\`ag function $h$ which
satisfies
%
%
\begin{equation}
\label{genLampTrans} {h ( t )}={f \biggl( \int_0^t
{h ( s )} \,ds \biggr)}+{g ( t )};
\end{equation}
a priori there might be many solutions.

When $g$ is identically equal to zero, a solution is found by the
method of time-changes: let $\tau$ be the first hitting time of zero
by $f$, let
\[
i_t=\int_0^t \frac{1}{{f ( s\wedge\tau)}} \,ds\vadjust{\goodbreak}
\]
and consider its right-continuous inverse $c$ so that
\[
h=f\circ c
\]
satisfies (\ref{genLampTrans}) with $g=0$, and it is the only solution
for which zero is absorbing. A generalization of this argument is found
in \citet{MR838085}, Chapter 6, Section 1. In this case the
transformation which takes $f$ to $h$ is called the Lamperti
transformation, introduced by \citet{1967LampertiCSBP}. There is a
slight catch: if $f$ is never zero and goes to infinity, then $h$
exists up to a given time (which might be infinite) when it also goes
to infinity. After this time, which we call the explosion time, we set
$h=\infty$. With this definition, note that $c$ and $h$ become
infinite at the same time.

Solutions to (\ref{genLampTrans}) are not unique even when $g=0$ as
the next example shows: take ${f ( x )}=\sqrt{
\llvert1-x\rrvert}$, $l>0$,
and consider
\[
{h_1 ( t )}=\frac{ ( 2-t )^+}{2} \quad\mbox{and}\quad {h_2 ( t )}=
\cases{ \dfrac{2-t}{2}, &\quad if $t\leq2$,
\vspace*{2pt}\cr
0, &\quad if $2\leq t\leq2+l$,
\vspace*{2pt}\cr
\dfrac{t-2-l}{2}, &\quad if $t\geq2+l$.}
\]
Then $h_1$ and $h_2$ are both solutions to (\ref{genLampTrans}). As we
discussed in the \hyperref[introSection]{Introduction},
a~probabilistically relevant example of nonuniqueness is obtained when
$g=0$ and $f$ is the typical sample path of a normalized Brownian
excursion $e= ( e_t,t\geq1 )$. [See Chapter 11, Section 3 of
\citet{MR1725357} for its definition as a \mbox{$3$-dimensional} Bessel
bridge.] Indeed, with probability $1$, $e$ has a continuous trajectory
which is positive exactly on $(0,1)$. Hence, $0$ is a solution to
${\ivp( e,0 )}$. However, its link with the $3$-dimensional Bessel
process (and time reversal) allows one to prove that $\sqrt{s}={o ( e_s
)}$ as $s\to0+$ (and a corresponding statement as $s\to1{-}$) so that
almost surely
\[
\int_0^1\frac{1}{e_s} \,ds<\infty.
\]
Hence, one can define the Lamperti transform of $e$, which is a
nontrivial solution to ${\ivp( e,0 )}$. The Lamperti
transformation is well defined under more general excursion laws as
discussed by \citet{MR2018924}.

We propose to prove Theorem~\ref{ExistenceTheorem} by the following
method: we first use the solution for the case $g=0$ to establish the
theorem when $g$ is piecewise constant. When $g$ is strictly
increasing, we approximate it by a strictly decreasing sequence of
piecewise constant functions $g_n> g$ and let $h_n$ be the solution to
(\ref{genLampTrans}) which uses~$g_n$. We then consider the primitive
$c_n$ of $h_n$ starting at zero, show that it converges, and this is
enough to prove the existence of a function whose right-continuous
derivative exists and solves (\ref{genLampTrans}). Actually, it is by
using primitives that one can compare the different solutions to (\ref
{genLampTrans}) (and study uniqueness), and this is the point of view
adopted in what follows. To this end, we generalize\vadjust{\goodbreak} (\ref
{genLampTrans}) into an initial value problem for the function $c$.
\[
{\ivp( f,g,x )}=\cases{ {c'_+ ( t )}={f\circ c ( t )}+{g ( t )},
\vspace*{1pt}\cr
{c ( 0 )}=x.}
\]
[The most important case for us is $x=0$, and we will write ${\ivp
( f,g )}$ when referring to it.] We shall term:
\begin{itemize}
\item$f$ the \textit{reproduction function},
\item$g$ the \textit{immigration function},
\item$x$ the \textit{initial cumulative population},
\item$c$ the \textit{cumulative population}, and
\item$c'_+$ the \textit{population profile}.
\item A solution $c$ to ${\ivp( f,g,x )}$ is said to have
\textit
{no spontaneous generation} if the condition ${c'_+ ( t
)}=0$ implies
that ${c ( t+s )}={c ( t )}$ as long as ${g
( t+s )}={g ( t )}$.
\end{itemize}
In the setting of Theorem~\ref{ExistenceTheorem}, spontaneous
generation is only relevant when $g$ is piecewise constant, and it will
be the guiding principle to chose solutions in this case.

A solution to ${\ivp( f,g,x )}$ without spontaneous
generation when
$g$ is a constant $\gamma$ is obtained by setting ${f_x (
s )}={f ( x+s )}+\gamma$, calling $h_x$ the Lamperti
transform of $f_x$ and setting
\[
c_t=x+\int_0^t {h_x (
s )} \,ds.
\]
We then have
\[
{c'_+ ( t )}={h_x ( t )}={f_x \biggl(
\int_0^t {h_x ( s )} \,ds \biggr)}={f
\biggl( x+\int_0^t {h_x ( s )} \,ds
\biggr)}+\gamma={f \bigl( {c ( t )} \bigr)}+{g ( t )}.
\]

Let $g$ be piecewise constant, say
\[
g=\sum_{i=1}^n \gamma_i
\si_{[t_{i-1},t_i)}
\]
with $\gamma_1<\gamma_2<\cdots<\gamma_n$ and $0=t_0<t_1<\cdots
<t_n$. Let us solve (\ref{genLampTrans}) by pasting the solutions on
each interval: let $\psi_1$ solve ${\ivp( f,\gamma_1,0
)}$ on
$[0,t_1]$ without spontaneous generation. Let $c$ equal $\psi_1$ on
$[0,t_1]$. Now, let $\psi_2$ solve ${\ivp( f,\gamma_2,{c
( t_1 )} )}$ without spontaneous generation. [If ${c
( t_1 )}=\infty$,
we set $\psi_2=\infty$.] Set ${c ( t )}={\psi_2 (
t-t_1 )}$ for
$t\in[t_1,t_2]$ so that $c$ is continuous. Also, for $t\in[t_1,t_2]$,
we have
\[
{c'_+ ( t )}={\psi'_{2+} (
t-t_1 )}={f \bigl( {\psi_2 ( t-t_1 )}
\bigr)}+\gamma_2={f \bigl( {c ( t )} \bigr)}+{g ( t )}.
\]
We continue in this manner. Note that if $c'_+$ reaches zero in
$[t_{i-1},t_i)$, say at $t$, then $c$ is constant on $[t,t_i)$ and that
$c'_+$ solves (\ref{genLampTrans}) when $g$ is piecewise constant. By
uniqueness of solutions to (\ref{genLampTrans}) which are absorbing at
zero when $g=0$, we deduce the uniqueness of solutions to ${\ivp
( f,g,0 )}$ without spontaneous generation when the immigration is
piecewise constant.\vadjust{\goodbreak}

We first tackle the nonnegativity assertion of Theorem \ref
{ExistenceTheorem}. Since $f$ is only defined on $[0,\infty)$,
negative values of $c$ do not make sense in equation (\ref
{genLampTrans}). One possible solution is to extend $f$ to $\re$ by
setting ${f ( x )}={f ( 0 )}$ for $x\leq0$.
%
%
\begin{lem}
\label{NonNegativityLemma}
Any solution $h$ to (\ref{genLampTrans}) is nonnegative.
\end{lem}
\begin{pf}
Let $h$ solve (\ref{genLampTrans}) where $f$ is extended by constancy
on $(-\infty,0]$, and define
\[
{c ( t )}=\int_0^t{h ( s )} \,ds,
\]
so that $c$ solves ${\ivp( f,g )}$. 
We prove that $h\geq0$ by contradiction. Assume there exists $t\geq0$
such that ${h ( t )}<0$. Note that since $h$ has no
negative jumps,
$h$ can only reach negative values continuously, and, since $h$ is
right-continuous, if it is negative at a given $t$, then there exists
$t'>t$ such that $h$ is negative on $[t,t')$. Hence there exists $\eps
>0$ such that
\[
\bigl\{ t\geq0\dvtx{h ( t )}=0\mbox{ and }h<0\mbox{ on }(t,t+\eps)
\bigr\}
\neq
\varnothing.
\]
Let $\tau$ be its infimum. We assert that $\tau>0$ and ${c (
\tau)}>0$. Indeed, if $\tau=0$, then $c$ would be strictly
decreasing and
negative on $(0,\eps)$, which would imply that
\[
{h ( t )}={f\circ c ( t )}+{g ( t )}={f ( 0 )}+{g ( t )}\geq{f ( 0
)}+{g ( 0 )}
\geq0 \qquad\mbox{for } t\in(0,\eps),
\]
a contradiction. A similar argument tells us that ${c ( \tau
)}>0$.
We finish the proof by showing the existence of $t_1\leq\tau$ and
$t_2\in(\tau,\tau+\eps)$ such that ${h ( t_1 )}>0$ and
${c ( t_1 )}={c ( t_2 )}$, implying the contradiction
\[
0< {h ( t_1 )}={f\circ c ( t_1 )}+{g ( t_1
)}={f\circ c ( t_2 )}+{g ( t_1 )}\leq{f\circ c (
t_2 )}+{g ( t_2 )}\leq0.
\]
Indeed, given that ${c ( \tau)}>0$ we can assume that ${c
( \tau+\eps)}>0$ by choosing a smaller~$\eps$, and then
let $\tau_1$ be the last time before $\tau$ that $c$ is below ${c ( \tau
+\eps)}$ and $\tau_2$ the first instant after $\tau_1$ that $c$
equals ${c ( \tau)}$. Note that $\tau_2\leq\tau$.
Since
\[
\int_{\tau_1}^{\tau_2}{h ( r )} \,dr={c (
\tau_2 )}-{c ( \tau_1 )}={c ( \tau)}-{c ( \tau+\eps)}>0,
\]
there exists $r\in(\tau_1,\tau_2)$ such that $h ( r
)>0$ and by
construction ${c ( r )}\in c ( (\tau,\break\tau+\eps) )$.
\end{pf}

%
\subsection{Monotonicity and existence}
We now establish a basic comparison lemma for solutions to ${\ivp
( f,g )}$ which will lead to the existence assertion of
Theorem \ref
{ExistenceTheorem}.
%
%
\begin{lem}
\label{strictMonotonicityOfIVPWRTImmigration}
Let $c$ and $\tilde c$ solve ${\ivp( f,g )}$ and $\ivp
(\tilde
f,\tilde g)$
. If
\[
g ( 0 )+{f ( 0 )}<\tilde g ( 0 )+{\tilde f ( 0 )},\qquad f\leq\tilde f, g\leq
\tilde g
\]
and either $g_-<\tilde g_-$ or $f_-<\tilde f_-$, then $c_t<\tilde c_t$
for every $t$ that is strictly positive and strictly smaller than the
explosion time of $c$.\vadjust{\goodbreak}
\end{lem}
It is important to note that the inequality $c\leq\tilde c$ cannot be
obtained from the hypothesis $g\leq\tilde g$ using the same
reproduction function $f$. Indeed, we would otherwise have uniqueness
for ${\ivp( f,g )}$ which, as we have seen, is not the
case even
when $g=0$. Also, since both $c$ and $\tilde c$ begin at $0$ and equal
$\infty$ after their explosion time, we always have the inequality
$c\leq\tilde c$ under the conditions of Lemma \ref
{strictMonotonicityOfIVPWRTImmigration}. 
%
\begin{pf*}{Proof of Lemma~\ref{strictMonotonicityOfIVPWRTImmigration}}
Let $\tau=\inf\{ t>0\dvtx{c ( t )}={\tilde c (
t )} \}$. Since
\[
{c'_+ ( 0 )}={f ( 0 )}+{g ( 0 )}<{f ( 0 )}+{\tilde g ( 0 )}={
\tilde c'_+ ( 0 )},
\]
and the right-hand derivatives of $c$ and $\tilde c$ are
right-continuous, then $\tau>0$ and $c<\tilde c$ on $(0,\tau)$. Note
then that the explosion time of $c$ cannot be smaller than $\tau$,
since this would force $\tilde c$ to explode before $\tau$ and so $c$
would equal $\tilde c$ before $\tau$.

We now argue by contradiction. If $\tau$ were finite, we know that
\[
{c ( \tau)}={\tilde c ( \tau)},
\]
leaving us with two cases,
\[
{c ( \tau)}={\tilde c ( \tau)}=\infty
\quad\mbox{and}\quad {c ( \tau)}={\tilde
c ( \tau
)}<\infty.
\]
In the former, we see that $\tau$ is the explosion time of $c$ and so
the statement of Lemma~\ref{strictMonotonicityOfIVPWRTImmigration}
holds. In the latter case,
\begin{eqnarray*}
{c'_- ( \tau)}&=&{f \bigl( c ( \tau)- \bigr)}+{g ( \tau- )}={f
\bigl( {\tilde c ( \tau)}- \bigr)}+{g ( \tau- )}\\
&<&{\tilde f \bigl(
{\tilde c (
\tau)}- \bigr)}+{\tilde g ( \tau- )}={\tilde c'_- ( \tau)}.
\end{eqnarray*}
It follows that $c'_-<\tilde c'_-$ in some interval $(\tau-\eps,\tau
)$. However, for $0<t<\tau$, we have ${c ( t )}<{\tilde c
( t )}$,
and this implies the contradiction
\[
{c ( \tau)}<{\tilde c ( \tau)}.
\]
\upqed\end{pf*}
\begin{pf*}{Proof of Theorem~\ref{ExistenceTheorem}, Existence}
Consider a sequence of piecewise constant c\`adl\`ag functions $g_n$
satisfying ${g_{n+1} ( 0 )}<{g_n ( 0 )}$,
$g_{n+1-}<g_{n-}$ and
such that $g_n\to g$ pointwise. Let $c_n$ solve ${\ivp(
f,g_n )}$
with no spontaneous generation. By Lemma \ref
{strictMonotonicityOfIVPWRTImmigration}, the sequence of nonnegative
functions $c_n$ is decreasing, so that it converges to a limit $c$. Let
\[
\tau=\inf\bigl\{ t\geq0\dvtx{c ( t )}=\infty\bigr\}=\liminf_{n\to
\infty}\bigl
\{ t\geq0\dvtx{c_n ( t )}=\infty\bigr\}.
\]
Since $f$ is right-continuous and $c<c_n$, $f\circ c_n+g_n$ converges
pointwise to $f\circ c+g$ on $[0,\tau)$. By bounded convergence, for
$t\in[0,\tau)$,
\[
{c ( t )}=\lim_{n\to\infty}{c_n ( t )}=\lim_{n\to\infty}\int
_0^t {f\circ c_n ( s
)}+{g_n ( s )} \,ds=\int_0^t {f\circ
c ( s )}+{g ( s )} \,ds.
\]
Hence, $h=c'_+$ proves the existence part of Theorem~\ref{ExistenceTheorem}.
\end{pf*}

\subsection{Uniqueness}
To study uniqueness of ${\ivp( f,g )}$, we use the
following lemma.\vspace*{-2pt}
%
%
\begin{lem}
\label{strictlyIncreasingImmigrationLemma}
If $g$ is strictly increasing, and $c$ solves $\ivp( f,g
)$, then
$c$ is strictly increasing.\vspace*{-2pt}
\end{lem}
\begin{pf}
Note that by Lemma~\ref{NonNegativityLemma}, the right-hand derivative
of $c$ is nonnegative, so that $c$ is nonnegative and nondecreasing.
By contradiction, if $c$ had an interval of constancy $[s,t]$, with
$t>s$, then
\begin{eqnarray*}
0 &=& {c'_+ \biggl( \frac{t+s}{2} \biggr)}
\\
&=& {f\circ c \biggl( \frac{t+s}{2} \biggr)}+{g \biggl( \frac
{t+s}{2}
\biggr)}
\\
&>& {f\circ c ( s )}+{g ( s )}
\\
&=& 0.\vspace*{-2pt}
\end{eqnarray*}
\upqed\end{pf}
\begin{remark*}
As we shall see in the proof of the uniqueness assertion of
Theorem~\ref{ExistenceTheorem}, if we can guarantee that all solutions
to ${\ivp( f,g )}$ are strictly increasing, then uniqueness holds for
${\ivp( f,g )}$. Note that if $f+{g ( 0 )}$ is strictly positive, then
${f ( x )}+{g ( t )}>0$ for all $x\geq0$ and $t\geq0$, so that all
solutions to ${\ivp( f,g )}$ are strictly increasing.\vspace*{-2pt}
\end{remark*}
\begin{pf*}{Proof of Theorem~\ref{ExistenceTheorem}, Uniqueness}
Let $c$ and $\tilde c$ solve ${\ivp( f,g )}$.\vspace*{1pt} To show that
$c=\tilde c$, we argue by contradiction by studying their inverses $i$
and~$\tilde i$.

Suppose that $c$ and $\tilde c$ are strictly increasing. Then $i$ and
$\tilde i$ are continuous. If $c\neq\tilde c $, then $ i\neq\tilde
i$, and we might\vspace*{1pt} without loss of generality suppose there is $x_1$ such
that ${i ( x_1 )}<{\tilde i ( x_1 )}$. Let
\[
x_0=\sup\bigl\{ x\leq x_1\dvtx{i ( x )}\geq{\tilde i ( x
)} \bigr\},
\]
and note that, by continuity of $ i$ and $\tilde i$, $x_0<x_1$ and $
i\leq\tilde i$ on $(x_0,x_1]$. Since $ i$ and $\tilde i$ are
continuous, they satisfy
\[
{i ( y )}=\int_0^y\frac{1}{{f ( x )}+{g\circ
i ( x )}} \,dx.
\]
There must exist $x\in[x_0,x_1]$ such that ${i' ( x )}$
and ${\tilde i' ( x )}$ both exist, and the former is
strictly smaller since
otherwise the inequality $\tilde i\leq i$ would hold on $[x_0,x_1]$.
For this value of $x$,
\[
{f ( x )}=\frac{1}{{\tilde i' ( x )}}-{g\circ\tilde i ( x )}<\frac
{1}{{i' ( x )}}-{g\circ
i ( x )}={f ( x )},
\]
which is a contradiction.

Note that all solutions to ${\ivp( f,g )}$ are strictly increasing
whenever $g$ is strictly increasing (by Lemma \ref
{strictlyIncreasingImmigrationLemma}) or $f$ is strictly positive,
which implies uniqueness to ${\ivp( f,g )}$ in these cases.\vadjust{\goodbreak}

When $g$ is constant, and $f+g$ is absorbed at $0$, meaning that if ${f
( s )}+{g ( 0 )}=0$, then ${f ( t )}+{g ( 0 )}=0$ for all $t\geq s$, we
can directly use the Lamperti transformation to obtain uniqueness.
Indeed, solutions to ${\ivp( f,g )}$ do not have spontaneous generation
and, as stated in the introduction to Section
\ref{ODESection} (cf. page \pageref{ODESection}), there is an unique
solution to ${\ivp( f+{g ( 0 )},0 )}$ without spontaneous generation.
\end{pf*}
%
\subsection{Uniqueness in the stochastic setting}
We now verify that solutions to (\ref{genLampRepDef}) are unique even
if the subordinator $Y$ is compound Poisson.
\begin{pf*}{Proof of Proposition~\ref{uniquenessForLevy}}
Let $X$ be a spLp and $Y$ an independent subordinator. We first prove
that there is an unique process $Z$ which satisfies
\[
Z_t=x+{X \biggl( \int_0^t
Z_s \,ds \biggr)}+Y_t.
\]

When $Y$ is an infinite activity subordinator (its L\'evy measure is
infinite or equivalently it has jumps in any nonempty open interval) or
it has positive drift, then its trajectories are strictly increasing,
and so uniqueness holds, thanks to Theorem~\ref{ExistenceTheorem}.

It then suffices to consider the case when $Y$ is a compound Poisson
process. There is a simple case we can establish: if $X$ is also a
subordinator, and $x>0$, then all solutions to ${\ivp(
X,x+Y )}$
are strictly increasing, and so uniqueness holds (again by Theorem \ref
{ExistenceTheorem}). It remains to consider two cases: when $X$ is a
subordinator and $x=0$ and when $X$ is not a subordinator. In the
first, note that zero solves ${\ivp( X,0 )}$, and since every
solution is nonnegative, zero is the smallest one. To prove
uniqueness, let $C^x$ be the (unique) solution to ${\ivp(
X,x )}$,
so that $C^x$ is greater than any solution to ${\ivp( X,0
)}$ by
Lemma~\ref{strictMonotonicityOfIVPWRTImmigration}. If we prove that as
$x\to0$, $C^x\to0$, then all solutions to ${\ivp( X,0
)}$ are
zero, and so uniqueness holds. For this, use the fact that as $t\to0$,
$X_t/t$ converges almost to the drift coefficient of $X$, say $d\in
[0,\infty)$ [cf. \citet{MR1406564}, Chapter III, Proposition 8, page
84] so that
\[
\int_{0+}\frac{1}{X_s} \,ds=\infty.
\]
Let $I^x$ be the (continuous) inverse of $C^x$ (note that $C^x$ is
strictly increasing). Since
\[
{I^x ( t )}=\int_0^t
\frac{1}{x+X_s} \,ds,
\]
we see, by Fatou's lemma, that $I^x\to\infty$ as $x\to0$, so that
$C^x\to0$. Now with $X$ still a subordinator and $Y$ compound Poisson,
the preceding case implies that the solution to ${\ivp(
X,Y )}$ is
unique until the first jump time of $Y$; after this jump time, all
solutions are strictly increasing, and hence uniqueness holds.\vadjust{\goodbreak}

The only remaining case is when $Y$ is compound Poisson and $X$ is not
a subordinator. The last hypothesis implies that $0$ is regular for
$(-\infty,0)$, meaning that on every interval $[0,\eps)$, $X$ visits
$(-\infty,0)$; cf. \citet{MR1406564}, Chapter VII, Theorem~1, page
189. From this, it follows that if $T$ is any stopping time with
respect to the filtration $\sigma( X_s,s\leq t )\vee
\sigma( Y ),t\geq0$,
then $X$ visits $(-\infty,X_{T})$ on any interval to the right of $T$.
Let $C$ be any solution to ${\ivp( X,x+Y )}$; we will
show that it
has no spontaneous generation. Since there is an unique solution
without spontaneous generation when $Y$ is piecewise constant (as
discussed in the introduction
to Section~\ref{ODESection}), we get
uniqueness. Let
\[
[T_{i-1},T_i),\qquad i=1,2,\ldots,
\]
be the intervals of constancy of $Y$; if $C$ has spontaneous generation
on one of these, say $[T_{i-1},T_i)$, then $X$ reaches the level
$-Y_{T_{i-1}}$ and then increases, which we know does not happen since
the hitting time of $ \{ -Y_{T_{i-1}} \}$ by the process $X$
is a
stopping time with respect to the filtration $\sigma( X_s,s\leq
t )\vee
\sigma( Y ), t\geq0$.

We end the proof by showing that any c\`adl\`ag process $Z$ satisfying
%
%
\begin{equation}
\label{DifferentialInequality} x+{X_- \biggl( \int_0^t
Z_s \,ds \biggr)}+Y_t\leq Z_t\leq x+{X \biggl(
\int_0^t Z_s \,ds
\biggr)}+Y_t
\end{equation}
actually satisfies
\[
Z_t= x+{X \biggl( \int_0^t
Z_s \,ds \biggr)}+Y_t.
\]
Let
\[
C_t=\int_0^t Z_s \,ds.
\]

When $Y$ is strictly increasing, an argument similar to the proof of
the Monotonicity lemma (Lemma \ref
{strictMonotonicityOfIVPWRTImmigration}) tells us that $C$ is strictly
increasing, so that $C$ actually satisfies ${\ivp( X,x+Y )}$.

When $Y=0$, the previous argument shows that, as long as $Z$ has not
reached $0$, $C$ coincides with the solution to ${\ivp(
X,x )}$. If
$Z$ is such that
\[
\inf\{ t\geq0\dvtx Z_t=0 \}=\inf\{ t\geq0\dvtx Z_{t-}=0 \},
\]
then $C$ solves ${\ivp( X,x )}$, which has an unique
solution, so
that (\ref{DifferentialInequality}) has an unique solution. We then
see that the only way in which $Z$ can cease to solve ${\ivp(
X,x )}$ is if $X$ is such that
\[
T_{0+}=\inf\{ t\geq0\dvtx x+ X_{t-}=0 \}<\inf\{ t\geq0\dvtx
x+X_t=0 \}=T_0,
\]
which is ruled out almost surely by quasi left-continuity of $X$.
Indeed, $T_{0+}$ is the increasing limit of the stopping times
\[
T_\eps=\inf\{ t\geq0\dvtx x+X_t<\eps\},
\]
which satisfy $T_{\eps}<T_{\eps'}$ if $\eps<\eps'$ since $X$ has no
negative jumps. Hence $X$ is almost surely continuous at $T_{0+}$ which
says that $x+X_{T_{0+}}=0$ almost surely. In the remaining case when
$Y$ is a (nonzero) compound Poisson process, we condition on $Y$ and
argue similarly on constancy intervals of $Y$.
\end{pf*}
%
\subsection{Explosion}

We now turn to the explosion criteria of solutions of ${\ivp(
f,g )}$ of Proposition~\ref{DeterministicExplosionProposition}.
\begin{pf*}{Proof of Proposition~\ref{DeterministicExplosionProposition}}
(1) If $\int^\infty1/{f^+ ( x )}=\infty$, let
$c$ be any
solution to ${\ivp( f,g )}$. We show that $c$ is finite
at every
$t>0$. Indeed, using the arguments of Lemma \ref
{strictMonotonicityOfIVPWRTImmigration}, we see that $c$ is bounded by
any solution to ${\ivp( f,1+{g ( t )} )}$ on
the interval
$[0,t]$. A particular solution to ${\ivp( f,1+{g (
t )} )}$ is
obtained by taking the right-continuous inverse of
\[
y\mapsto\int_0^y \frac{1}{{f ( x )}+1+{g ( t
)}} \,dx.
\]
Since
\[
\int_0^\infty\frac{1}{{f^+ ( x )}+1+{g ( t
)}} \,dx=\infty,
\]
the particular solution we have considered is everywhere finite.

(2) Let $c$ be a solution to ${\ivp( f,g )}$
where $f$ is
an explosive reproduction function, $\lim_{x\to\infty}{f (
x )}=\infty$ and ${g ( \infty)}$ exceeds the
maximum of $-f$. To
prove that $c$ explodes, choose $T>0$ such that ${f ( x
)}+{g ( t )}>0$ for all $x\geq0$ and $t\geq T$. Then
$f\circ c+g>0$ on
$[T,\infty)$. Let $M={c ( T )}$. We then consider the
right-continuous inverse $i$ of $c$ (which is actually an inverse on
$[M,\infty)$) and note that for $y>M$,
\[
{i ( y )}-{i ( M )}=\int_M^y
\frac{1}{{f
( x )}+{g\circ i ( x )}} \,dx\leq\int_M^y
\frac
{1}{{f ( x )}} \,dx.
\]
Hence, ${i ( y )}$ converges to a finite limit as $y\to
\infty$ so
that $c$ explodes.
%
\end{pf*}
%
\subsection{Application of the analytic theory}
We now pass to a probabilistic application of Theorem~\ref{ExistenceTheorem}.
\begin{pf*}{Proof of Corollary~\ref{LambertProblemCorollary}}
We consider first the case where $Y$ is deterministic. Since $Y$ is
assumed to be strictly increasing, we can consider the unique
nonnegative stochastic process $Z$ which satisfies
\[
Z_t=x+X_{\int_0^t Z_s \,ds}+Y_t.
\]
(The reader can be reassured by Lemma~\ref{MeasurabilityDetailsLemma}
regarding any qualms on measurability issues.) Since $Z$ is
nonnegative, Theorems 4.1 and 4.2 of \citet{MR1158024} imply the
existence of a stochastic process $\tilde X$ with the same law as $X$
such that
\[
Z_t=x+\int_0^t
Z_s^{1/\alpha} \,d\tilde X_s+Y_t.
\]
Hence $Z$ is a weak solution to (\ref{SDEStablespLp}). 

Conversely, if $Z$ is a solution to (\ref{SDEStablespLp}), we apply
Theorems 4.1 and 4.2 of \citet{MR1158024} to deduce the existence
of a
stochastic process $\tilde X$ with the same law as $X$ such that
\[
Z_t=x+\tilde X_{\int_0^t Z_s \,ds}+Y_t.
\]
Considering the mapping $ ( f,g )\mapsto{F (
f,g )}$ that
associates to every admissible breadth-first pair the solution $h$ to
(\ref{genLampTrans}), we see that $Z$ has the law of ${F (
\tilde X,x+Y )}$. Hence, weak uniqueness holds for (\ref
{SDEStablespLp}).

When $Y$ is not deterministic but independent of $X$, we just reduce to
the previous case by conditioning on $Y$ [or by augmenting the
filtration with the $\sigma$-field $\sigma( Y_t\dvtx t\geq0 )$].
\end{pf*}

\section{CBI processes as Lamperti transforms}
\label{CBIRepSection}

We now move on to the analysis of Theorem~\ref{CBIRepThm}. Let $X$ and
$Y$ be independent L\'evy processes such that $X$ is spectrally
positive and $Y$ is a subordinator under the probability measure~$\p$.
Call $\Psi$ and $\Phi$ their Laplace exponents (taking care to have
$\Phi\geq0$ as for subordinators). Note that the trajectories of $Y$
are either zero, piecewise constant (in the compound Poisson case), or
strictly increasing.

Let $Z$ be the stochastic process that solves
\[
Z_t=x+X_{\int_0^t Z_s \,ds}+Y_t
\]
and has no spontaneous generation (when $Y$ is compound Poisson). To
prove that $Z$ is a ${\cbi( \Psi,\Phi)}$, we should
see that it
is a c\`adl\`ag and homogeneous Markov process and that there exist
functions ${u_t\dvtx(0,\infty)\to(0,\infty)}$ and ${v_t\dvtx
(0,\infty)\to(0,\infty)}$, satisfying
%
%
\begin{equation}
\label{DifferentialEquationsForCBIExponents}
\cases{ \dfrac{\partial
}{\partial t}{u_t
( \lambda)}=-{\Psi\circ u_t ( \lambda)},
\vspace*{2pt}\cr
{u_0 (
\lambda)}=\lambda,} \quad\mbox{and}\quad
\cases{ \dfrac{\partial}{\partial
t}{v_t (
\lambda)}={\Phi\bigl( {u_t ( \lambda)} \bigr)},
\vspace*{2pt}\cr
{v_0
( \lambda)}=0,}
\end{equation}
and such that for all $\lambda,t\geq0$,
\[
{\se\bigl( e^{-\lambda Z_t} \bigr)}=e^{-x{u_t ( \lambda
)}-{v_t ( \lambda)}.}
\]
[At this point it should be clear that the equation for $u$
characterizes it and that, actually, for fixed $\lambda>0$, $t\mapsto
{u_t ( \lambda)}$ is the inverse function to
\[
x\mapsto\int_x^\lambda\frac{1}{{\Psi( y )}} \,dy.\mbox{$\bigg]$}
\]
%

\subsection{A characterization lemma
and a short proof of Lamperti's theorem}
The way to compute the Laplace transform of $Z$ is by showing, with
martingale arguments to be discussed promptly,
that
%
%
\begin{equation}
\label{ExpectationForSolToIVP} {\se\bigl( e^{-\lambda Z_t} \bigr)}=\int
_0^t{\se\bigl( \bigl[ {\Psi( \lambda
)}Z_s-{\Phi( \lambda)} \bigr]e^{-\lambda Z_s} \bigr)} \,ds.
\end{equation}
We are then in a position to apply the following result.
%
%
\begin{lem}[(Characterization lemma)]
\label{CharacterizationLemma}
If $Z$ is a nonnegative homogeneous Markov process with c\`adl\`ag
paths starting at $x$ and satisfying (\ref{ExpectationForSolToIVP})
for all $\lambda>0$, then $Z$ is a ${\operatorname{CBI} ( \Psi,\Phi)}$ that
starts at $x$.
\end{lem}
\begin{remark*}
Note that the hypotheses on the process $Z$ of Lemma \ref
{CharacterizationLemma} do not allow us to use generator arguments
which would shorten the proof, for example, by using the
characterization of the infinitesimal generator of a CBI process
through exponential functions.
%
\end{remark*}
\begin{pf*}{Proof of Lemma~\ref{CharacterizationLemma}}
Let us prove that the function
\[
{G ( s )}=\se\bigl( e^{-{u_{t-s} ( \lambda
)} Z_s-{v_{t-s} ( \lambda)}} \bigr)
\]
satisfies ${G' ( s )}=0$ for $s\in[0,t]$, so that it is
constant on
$[0,t]$, implying the equality
\[
\se\bigl( e^{-\lambda Z_t} \bigr)={G ( t )}={G ( 0 )}=e^{-x{u_t (
\lambda)}-{v_t ( \lambda
)}}.
\]
We then see that $Z_t$ has the same one-dimensional distributions as a
${\cbi( \Psi,\Phi)}$ that starts at $x$, so that by
the Markov
property, $Z$ is actually a ${\cbi( \Psi,\Phi)}$.

To see that $G'=0$, we first write
%
%
\begin{eqnarray}
\label{DecompositionOfExpectationInCharacterizationLemma} {G ( s+h
)}-{G ( s )}
&=& \bigl( {G ( s+h )}-\se\bigl( e^{-{u_{t-s-h} (
\lambda)}Z_s-{v_{t-s-h} ( \lambda)}} \bigr) \bigr)
\nonumber\\[-8pt]\\[-8pt]
&&{} + \bigl( \se\bigl( e^{-{u_{t-s-h} ( \lambda
)}Z_s-{v_{t-s-h} ( \lambda)}} \bigr)-{G ( s )} \bigr).
\nonumber
\end{eqnarray}
We now analyze both summands to later divide by $h$ and let $h\to0$.

For the first summand, use (\ref{ExpectationForSolToIVP}) to get
\begin{eqnarray*}
&&{G ( s+h )}-\se\bigl( e^{-Z_s {u_{t-s-h} ( \lambda
)}-{v_{t-s-h} ( \lambda)}} \bigr)
\\
&&\qquad=e^{-{v_{t-s-h} ( \lambda)}} \int_s^{s+h} \se\bigl(
e^{-Z_{r}{u_{t-s-h} ( \lambda)}} \bigl[ Z_r {\Psi\circ u_{t-s-h} (
\lambda
)}-{\Phi\circ u_{t-s-h} ( \lambda)} \bigr] \bigr) \,dr,
\end{eqnarray*}
so that, since $Z$ has c\`adl\`ag paths, we get
\begin{eqnarray*}
&&\lim_{h\to0}\frac{1}{h} \bigl[ {G ( s+h )}-\se\bigl(
e^{-Z_s
{u_{t-s-h} ( \lambda)}-{v_{t-s-h} ( \lambda
)}} \bigr) \bigr]
\\
&&\qquad=\se\bigl( e^{-{u_{t-s} ( \lambda)}Z_{s}-{v_{t-s}
( \lambda)}} \bigl[ Z_s{\Psi\circ u_{t-s}
( \lambda)}-{\Phi\circ u_{t-s} ( \lambda)} \bigr] \bigr).
\end{eqnarray*}

For the second summand in the right-hand side of (\ref
{DecompositionOfExpectationInCharacterizationLemma}), we differentiate
under the expectation to obtain
\begin{eqnarray*}
&&\lim_{h\to0}\frac{1}{h}\se\bigl( e^{-{u_{t-s-h} ( \lambda
)}
Z_s-{v_{t-s-h} ( \lambda)}}-e^{-{u_{t-s} ( \lambda
)} Z_s-{v_{t-s} ( \lambda)}}
\bigr)
\\
&&\qquad=\se\biggl( e^{-{u_{t-s} ( \lambda)} Z_s-{v_{t-s}
( \lambda)}} \biggl[ Z_s\frac{\partial{u_{t-s} (
\lambda)}}{\partial s} +
\frac{\partial{v_{t-s} (
\lambda)}}{\partial
s} \biggr] \biggr).
\end{eqnarray*}
We conclude that ${G' ( s )}=0$ for all $s\in[0,t]$, using
(\ref
{DifferentialEquationsForCBIExponents}).
\end{pf*}

A simple case of our proof of Theorem~\ref{CBIRepThm} arises when
$Y=0$. Recall from Proposition~\ref{JumpingConditionProposition} the notation
\[
C_t=\int_0^t Z_s \,ds.
\]
\begin{pf*}{Proof of Theorem~\ref{CBIRepThm} when $\Phi=0$}
This is exactly the setting of Lamperti's theorem stated by \citet
{1967LampertiCSBP}.

When $\Phi=0$ (or equivalently, $Y$ is zero), then $C_t$ is a stopping
time for $X$ [since the inverse of $C$ can be obtained by integrating
$1/(x+X)$]. Since $Z$ is the time-change of $X$ using the inverse of an
additive functional, $Z$ is a homogeneous Markov process. [Another
proof of the Markov property of~$Z$, based on properties of ${\ivp
( X,x+Y )}$ is given in~\ref{HomogeneousMarkovProperty}
of Lemma \ref
{MeasurabilityDetailsLemma}.]
Also, we can transform the martingale
\[
e^{-\lambda X_t}-{\Psi( \lambda)}\int_0^te^{-\lambda
X_s}
\,ds
\]
by optional sampling into the martingale
\[
e^{-\lambda Z_t}-{\Psi( \lambda)}\int_0^te^{-\lambda
Z_s}Z_s
\,ds.
\]
We then take expectations and apply Lemma~\ref{CharacterizationLemma}.
\end{pf*}
%
\subsection{The general case}
For all other cases, we need the following measurability details.
Consider the mapping $F_t$ which takes a c\`adl\`ag function $f$ with
nonnegative jumps and starting at zero, a c\`adl\`ag $g$ starting at
zero (either piecewise constant or strictly increasing), and a
nonnegative real $x$ to ${c'_{+} ( t )}$ where $c$ solves
${\ivp( f,x+g )}$ and has no spontaneous generation (if
$g$ is piecewise
constant). [Note that these conditions uniquely specify a solution to
${\ivp( f,x+g )}$.]
Then
%
%
\begin{equation}
\label{EvolutionRuleForZ} Z_{t+s}={F_t (
X_{C_s+\cdot}-X_{C_s},Y_{s+\cdot}-Y_s,Z_s
)}.
\end{equation}
The mapping $F_t$ is measurable. Indeed, we can view it as the
composition of three measurable mappings. The first one is the mapping
that associates to $ ( f,g+x )$ the unique solution to
${\ivp( f,g )}$ (without spontaneous generation), from
the space of admissible
breadth-first pairs equipped with the $\sigma$-fields generated by the
projections $ ( f,g )\mapsto{f ( t )}$ and
$ (
f,g )\mapsto
{g ( t )}$ for any $t\geq0$ to the space of nondecreasing
continuous functions with c\`adl\`ag derivative (equipped also with the
$\sigma$-field generated by projections). This mapping is measurable
when $g=0$ by measurability of the Lamperti transformation. Next, when
$g$ is piecewise constant this follows by concatenation of Lamperti
transforms as in the introduction to Section
\ref{ODESection}, and for
strictly increasing $g$, this follows since the unique solution to
${\ivp( f,g )}$ is the limit of solutions to ${\ivp
( f,g_n )}$
with piecewise constant functions $g_n$, as seen in the proof of
Theorem~\ref{ExistenceTheorem}. The second mapping sends a continuous
function with c\`adl\`ag derivative to its derivative, which is
measurable by approximation of the derivative by a sequence of
differential quotients. The third mapping is simply the projection of a
c\`adl\`ag function to its value at time $t$; its measurability is
proved in Theorem 12.5, page 134 of \citet{MR1700749}.

We suppose that our probability space $ ( \Omega,\F,\p)$
is complete, and let
$\mathscr{T}$ stand for the sets in $\F$ of probability zero. For
fixed $y,t\in
[0,\infty]$, let $\G^t_y=\F^X_y\vee\F^Y_t\vee\mathscr{T}$.
%
%
\begin{lem}[(Measurability details)]
\label{MeasurabilityDetailsLemma}
\textup{(1)} The filtration $ ( \G^t_y,y\geq0 )$ satisfies the
usual hypotheses.\vspace*{-6pt}
{
\renewcommand\thelonglist{(\arabic{longlist})}
\renewcommand\labellonglist{\thelonglist}
\begin{longlist}\setcounter{longlist}{1}
\item $C_t$ is a stopping time for the filtration $ ( \G^t_y,y\geq
0 )$, and we can therefore define the $\sigma$-field
\[
\G^t_{C_t}= \bigl\{ A\in\F\dvtx A\cap\{ C_t\leq y
\}\in\G^t_y \bigr\}.
\]
\item\label{HomogeneousMarkovProperty} $Z$ is a homogeneous
Markov process with respect to the filtration $ ( \G^t_{C_t}, y\geq0 )$.
\end{longlist}}
\end{lem}
\begin{pf}
(1) We just need to be careful to avoid \textit{one of the worst
traps involving $\sigma$-fields} by using independence; cf. \citet
{2003ChaumontYor}, Example 25, page~29.

(2) We are reduced to verifying
%
%
\begin{equation}
\label{WeirdStoppingTimeProperty} \{ C_t< y \}\in
\G^t_y.
\end{equation}
We prove (\ref{WeirdStoppingTimeProperty}) in two steps, first when
$Y$ is piecewise constant, then when $Y$ is strictly increasing.

Let $Y$ be piecewise constant, jumping at the stopping times
$T_1<T_2<\cdots\,$, and set $T_0=0$. We first prove that
%
%
\begin{equation}
\label{WeirdStoppingTimePropertyAtStoppingTime} \{ C_{T_n}<y \}\cap\{
T_n\leq t \}\in\G_y^t
\end{equation}
and this result and a similar argument will yield (\ref
{WeirdStoppingTimeProperty}). The membership in (\ref
{WeirdStoppingTimePropertyAtStoppingTime}) is proved by induction using
the fact that $C$ can be written down as a Lamperti transform on each
interval of constancy of $Y$. Let $I_t$ be the functional on the
subspace of Skorohod space consisting of functions with nonnegative
jumps that aids in defining the Lamperti transformation: when applied
to a given function $f$, we first define
\[
{T_0 ( f )}=\inf\bigl\{ t\geq0\dvtx{f ( t )}=0 \bigr\}
\]
and then
\[
{I_t ( f )}=\int_0^{t\wedge{T_0 ( f )}}
\frac
{1}{{f ( s )}} \,ds.
\]
We then have
\[
\{ C_{T_1}<y \}\cap\{ T_1\leq t \}= \bigl\{
{I_y ( X+Y_0 )}>T_1\wedge t \bigr\}\cap\{
T_1\leq t \} \in\G_y^t.
\]
If we suppose that
\[
\{ C_{T_n}<y \}\cap\{ T_n\leq t \}\in
\G_y^t,
\]
%
then the decomposition
\begin{eqnarray*}
&& \{ C_{T_{n+1}}<y \}\cap\{ T_{n+1}\leq t \}
\\
&&\qquad=\bigcup_{q\in(0,y)\cap\ra}\bigcup_{m=1}^\infty
\bigcup_{k=0}^{2^{-m}\lfloor2^m ( y-q )\rfloor} \biggl\{
\frac{k}{2^m}\leq C_{T_n}< \frac{k+1}{2^m} \biggr\}\cap\{
T_{n+1}\leq t \}
\\
&&\hspace*{114pt}\qquad\quad{}\cap\bigl\{ {I_{q} ( x+X_{({{k}/{2^m}})+\cdot
}+Y_{T_n}
)}>T_{n+1}-T_n \bigr\}
\nonumber
\end{eqnarray*}
allows us to obtain (\ref{WeirdStoppingTimePropertyAtStoppingTime}).
Then the decomposition
\begin{eqnarray*}
\{ C_t<y \} &=&\bigcup_{n=0}^\infty
\bigcup_{q\in(0,y)\cap\ra}\bigcup_{m=1}^\infty
\bigcup_{k=0}^{2^{-m}\lfloor2^m ( y-q
)\rfloor} \{ T_n\leq
t<T_{n+1} \}\\
&&\qquad\quad\hspace*{97.1pt}{}\cap\biggl\{ \frac{k}{2^m}\leq C_{T_n}<
\frac{k+1}{2^m} \biggr\}
\\
&&\qquad\quad\hspace*{97.1pt}{}\cap\bigl\{ {I_q ( x+X_{({k}/{2^m})+\cdot
}+Y_{T_n}
)}>t-T_n \bigr\}
\end{eqnarray*}
gives (\ref{WeirdStoppingTimeProperty}) when $Y$ is piecewise constant.

When $Y$ is strictly increasing, consider a sequence $\eps_n$
decreasing strictly to zero and a decreasing sequence $ ( \pi_n )$
of partitions of $[0,t]$ whose norms tend to zero, with
\[
\pi_n= \bigl\{ t_0^n=0<t_1^n<
\cdots<t_{k_n}^n=t \bigr\}.
\]
Consider the process $Y^n= ( Y^n_s )_{s\in[0,t]}$ defined by
\[
Y^n_s=\eps_n+\sum
_{i=1}^{k_n}Y_{t_i^n}{\si_{[t_{i-1}^n,t_i^n)} ( s
)}+Y_t\si_{s=t}.
\]
Since $\pi^{n}$ is contained in $\pi^{n+1}$ and $\eps_n>\eps_{n+1}$,
$Y^n>Y^{n+1}$. If $C^n$ is the solution to ${\ivp(
X,x+Y^n )}$ with no spontaneous\vadjust{\goodbreak} generation (defined only on $[0,t]$),
then Lemma~\ref{strictMonotonicityOfIVPWRTImmigration} gives
$C^{n}>C^{n+1}$. Hence, $ ( C^n )$ converges as $n\to\infty
$, and
since the limit is easily seen to be a solution to ${\ivp(
X,x+Y )}$, the limit must equal $C$ by the uniqueness statement in
Theorem~\ref{ExistenceTheorem}. To obtain (\ref
{WeirdStoppingTimeProperty}), we note that
\[
\bigl\{ C^n_t<y \bigr\}\in\F^X_y
\vee\F^{Y^n}_t\subset\F^X_y\vee
\F^{Y}_t
\]
and
\[
\{ C_t<y \}=\bigcup_{n} \bigl\{
C^n_t<y \bigr\}.
\]

(3) Mimicking the proof of the Strong Markov Property for
Brownian motion [as in \citet{MR1876169}, Theorem 13.11] and using
(\ref{WeirdStoppingTimeProperty}), one proves that the process
\[
( X_{C_t+s}-X_{C_t},Y_{t+s}-Y_t
)_{s\geq0}
\]
has the same law as $ ( X,Y )$ and is independent of $\G^t_{C_t}$,
which we can restate as
\[
\begin{tabular}{p{320pt}}
$ ( X_{C_t+s}-X_{C_t},Y_{t+s}-Y_t )_{s\geq0}$ has the same
law as
$ ( X,Y )$ and\vspace*{2pt} is independent of $ ( X^{C_t},Y^t
)$ where
$X^{C_t}_s=X_{C_t\wedge s}$ and $Y^t_s=Y_{t\wedge s}$.
\end{tabular}
\]
Equation (\ref{EvolutionRuleForZ}) implies that the conditional law of
$Z_{t+s}$ given $\G^s_{C_s}$ is actually $Z_s$ measurable, implying
the Markov property. The transition semigroup is homogeneous and in $t$
units of time is given by the law ${P_t ( x,\cdot)}$ of
${F_t ( X,Y,x )}$ under $\p$. Note that this semigroup is
conservative on
$[0,\infty]$.
\end{pf}

We will need Proposition~\ref{JumpingConditionProposition} for our
proof of Theorem~\ref{CBIRepThm}.
\begin{pf*}{Proof of 
Proposition~\ref{JumpingConditionProposition}}
Consider the filtration $\G= ( \G_t )$ given by
\[
\G_t=\sigma( X_s\dvtx s\leq t )\vee\sigma(
Y_s\dvtx s\geq0 )\vee\mathscr{T}.
\]

If $Y$ is strictly increasing, then $C$ is strictly increasing and
continuous. For fixed $\eps>0$, let $T_1<T_2<\cdots$ be the jumps of
$Y$ of magnitude greater than $\eps$. Arguing as in Lemma \ref
{MeasurabilityDetailsLemma}, we see that $C_{T_i}$ is a $\G$-stopping
time which is the almost sure limit of the $\G$-stopping times
$C_{ ( T_i-1/n )^+}$ as $n\to\infty$. Since $X$ is a $\G
$-L\'evy
process and
$C_{ ( T_i-1/n )^+}<C_{T_i}$ for all $n$, quasi
left-continuity of
$X$ implies that $X\circ C$ does not jump at $T_i$ almost surely. Since
this is true for any $\eps>0$, then $X\circ C$ and $Y$ do not jump at
the same time.

If $Y$ is compound Poisson, we argue on its constancy intervals,
denoted $[T_{i-1},T_i)$, $i=1,2,\ldots\,$. On the set
$\{C_{s}<C_{T_i}$ for all $s<T_i\}$, we can argue as above, using
quasi left-continuity. On the set $\{ C_s=C_{T_i}$
for some $s<T_i\}$, we note that $X$ reaches $-Y_{T_i}$ for the first
time at
$C_{T_i}$. The hitting time of $-Y_{T_i}$ by $X$ is a $\G$-stopping
time which is the almost sure limit of the hitting times of
$-Y_{T_i}+1/n$ as $n\to\infty$. The latter are strictly smaller than
the former since $X$ has no negative jumps. Hence, by quasi
left-continuity, $X$ is almost surely continuous at $C_{T_i}$.\vadjust{\goodbreak}
\end{pf*}
\begin{pf*}{Proof of Theorem~\ref{CBIRepThm}}
Since
\[
\biggl( e^{-\lambda X_y}-{\Psi( \lambda)}\int_0^ye^{-\lambda
X_s}
\,ds \biggr)_{t\geq0}
\]
is a $ ( \G^t_y )_{y\geq0}$-martingale, it follows that
$M= ( M_t )_{t\geq0}$, given by
\[
M_t=e^{-\lambda X_{C_t}}-{\Psi( \lambda)}\int_0^t
e^{-\lambda
X_{C_s}}Z_s \,ds,
\]
is a $ ( \G^t_{C_t} )_{t\geq0}$-local martingale. With
respect to
the latter filtration, the stochastic process $N= ( N_t
)_{t\geq
0}$ given by
\[
N_t=e^{-\lambda Y_t}+{\Phi( \lambda)}\int_0^t
e^{-\lambda Y_s} \,ds
\]
is a martingale. Hence $e^{-\lambda X\circ C}$ and $e^{-\lambda(
x+Y )}$ are semimartingales to which
we may apply integration by parts
to get
\begin{eqnarray*}
e^{-\lambda Z_t} 
&=&\mbox{local martingale}+\int_0^t
e^{-\lambda Z_s} \bigl[ {\Psi( \lambda)}Z_s-{\Phi( \lambda)}
\bigr] \,ds\\
&&{}+ \bigl[ e^{-\lambda
X\circ C},e^{-\lambda x-\lambda Y} \bigr]_t,
\end{eqnarray*}
where the local martingale part is
\[
t\mapsto\int_0^{t} e^{-\lambda( x+Y_s )}
\,dM_s+\int_0^t e^{-\lambda X\circ C_s}
\,dN_s.
\]
Since $X\circ C$ and $Y$ do not jump at the same time by Proposition
\ref{JumpingConditionProposition} and $Y$ is of finite variation, we
see that
\[
\bigl[ e^{-\lambda X\circ C},e^{-\lambda x-\lambda Y} \bigr]=0;
\]
cf. \citet{MR1876169}, Theorem 26.6(vii).

We deduce that
\[
e^{-\lambda Z_t}-\int_0^t e^{-\lambda Z_s}
\bigl[ {\Psi( \lambda)}Z_s-{\Phi( \lambda)} \bigr] \,ds
\]
is a martingale, since it is a local martingale whose sample paths are
uniformly bounded on compacts thanks to the nonnegativity of $Z$.
Taking expectations, we get (\ref{ExpectationForSolToIVP}), and we
conclude by applying Lemma~\ref{CharacterizationLemma} since $Z$ is a
Markov process thanks to Lemma~\ref{MeasurabilityDetailsLemma}.
\end{pf*}

\subsection{Translating a law of the iterated logarithm}

\mbox{}

\begin{pf*}{Proof of Corollary~\ref{LIL}}
Let $X$ be a spLp with Laplace exponent $\Psi$, $\tilde\Phi$ be the
right-continuous inverse of $\Psi$ and
\[
{\alpha( t )}=\frac{{\log}\llvert{\log t}\rrvert}{{\tilde
\Phi( {t^{-1}\log}\llvert{\log t}\rrvert)}}.\vadjust{\goodbreak}
\]
Recall that $\tilde\Phi$ is the Laplace exponent of the subordinator
$T= ( T_t,t\geq0 )$ where
\[
T_t=\inf\{ s\geq0\dvtx X_s\leq-t \};
\]
cf. \citet{MR1406564}, Chapter VII, Theorem 1. If $\tilde d$ is the
drift coefficient of $\tilde\Phi$, then Proposition 1 of
Bertoin [(\citeyear{MR1406564}), Chapter III] gives
\[
\lim_{\lambda\to\infty}\frac{{\tilde\Phi( \lambda
)}}{\lambda}=\tilde d.
\]
Hence,
\[
\mbox{as $t\to0+$}\qquad \cases{ {\alpha( t )}\sim t/\tilde d, &\quad if
$\tilde d>0$,
\vspace*{1pt}\cr
t={o \bigl( {\alpha( t )} \bigr)}, &\quad if $\tilde d=0$.}
\]
We now assert that if $a_t\to1$ as $t\to0$, then
\[
\lim_{t\to0}\frac{{\alpha( a_t t )}}{{\alpha(
t )}}=1.
\]
This is clear when $\tilde d>0$, so suppose that $\tilde d=0$. Since
$t\mapsto{\log}\llvert{\log t}\rrvert$ is slowly varying at zero, it
suffices to
show that if $b_\lambda\to1$ as $\lambda\to\infty$, then
%
%
\begin{equation}
\label{LimitForLIL} \lim_{\lambda\to\infty}\frac{{\tilde\Phi(
b_\lambda
\lambda)}}{{\tilde\Phi( \lambda)}}=1.
\end{equation}
However, concavity of $\tilde\Phi$, increasingness and nonnegativity
give (if $b>1$)
\[
{\tilde\Phi( b\lambda)}/b\leq{\tilde\Phi( \lambda)}\leq{\tilde\Phi
( b
\lambda)},
\]
which implies (\ref{LimitForLIL}).

As noted by \citet{MR1312150}, \citet{MR0292163} prove the
existence of
a constant $\zeta\neq0$
such that
\[
\liminf_{t\to0}\frac{X_t}{{\alpha( t )}}=\zeta.
\]
%

Let $Z$ be the unique solution to
\[
Z_t=x+X_{\int_0^t Z_s \,ds}+Y_t
\]
with $x>0$, where $X$ and $Y$ are independent L\'evy processes, with
$X$ spectrally positive of Laplace exponent $\Psi$ and $Y$ a
subordinator with Laplace exponent $\Phi$. Since $Z_0=x$, and $Z$ is
right-continuous, then
\[
\lim_{t\to0+}\frac{1}{t}\int_0^t
Z_s \,ds=x
\]
almost surely. Hence
\[
\lim_{t\to0+}\frac{{\alpha( \int_0^t Z_s \,ds
)}}{{\alpha( xt )}}=1
\]
and so
\[
\liminf_{t\to0+} \frac{X_{\int_0^t Z_s \,ds}}{{\alpha(
xt )}}=\zeta. 
\]
On the other hand, if $d$ is the drift of $\Phi$, then
\[
\lim_{t\to0}\frac{Y_t}{t}=d
\]
[cf. \citet{MR1406564}, Chapter III, Proposition 8] so that if
$\tilde
d=0$, $Y_t={o ( {\alpha( t )} )}$ and
\[
\liminf_{t\to0+}\frac{Z_t-x}{{\alpha( xt )}}=\zeta.
\]
If $\tilde d>0$, then by Proposition 8 of
Bertoin [(\citeyear{MR1406564}), Chapter III], we actually have
\[
\liminf_{t\to0}\frac{X_t}{{\alpha( t )}}=-1
\]
so that
\[
\liminf_{t\to0+}\frac{Z_t-x}{{\alpha( xt )}}=-1+\frac
{d\tilde d}{x}.
\]
\upqed\end{pf*}

\subsection{Explosion criteria for CBI}
As a probabilistic application of the deterministic explosion criteria
of Proposition~\ref{DeterministicExplosionProposition}, we prove
Corollary~\ref{CBIExplosionCorollary}.
\begin{pf*}{Proof of Corollary~\ref{CBIExplosionCorollary}}
Let $x>0$, and consider a spectrally positive L\'evy process $X$ with
Laplace exponent $\Psi$ independent of a subordinator $Y$ with Laplace
exponent $\Phi$. Let $Z$ be the unique solution to
\[
Z_t=x+X_{\int_0^t Z_s \,ds}+Y_t,
\]
which is a ${\cbi( \Psi,\Phi)}$ that starts at $x$.
Also, let
\[
C_t=\int_0^t Z_s \,ds.
\]

\begin{longlist}[(3)]
\item[(1)] Let $Y$ be a nonzero subordinator. Path by path, we see
that $Z$ jumps to infinity if and only if either $X$ jumps to infinity
or $Y$ does. However, the probability that either $X$ or $Y$ jumps to
infinity is positive if and only if either ${\Psi( 0
)}>0$ or ${\Phi( 0 )}>0$. When $Y$ is zero, $Z$ jumps to
infinity if $X$ jumps to
infinity and never reaches $-x$, which has positive probability.

\item[(2)] The Ogura--Grey explosion criterion for continuous state
branching processes (as stated just before Corollary \ref
{CBIExplosionCorollary}) can be restated as follows: a ${\cbi(
\Psi,0 )}$ started at $x$ reaches $\infty$ continuously at a finite
time with positive probability if and only if ${\Psi( 0
)}=0$, and
$\Psi$ is an explosive branching mechanism. It is also simple to see
that a ${\cbi( \Psi,0 )}$ jumps to $\infty$ at a finite
time with
positive probability if and only if ${\Psi( 0 )}>0$.\vadjust{\goodbreak}

Path by path,\vspace*{1pt} we see that if $Z$ reaches $\infty$ continuously (say at
time $\tau$), then $Y$ does not jump to infinity on $[0,\tau)$. Also,
if we let $\tilde C$ be the unique\vspace*{1pt} solution to ${\ivp(
x+Y_{\tau-}+\eps+X,0 )}$ and $\tilde Z$ as the right-hand
derivative of $\tilde
C$, where $\eps>0$, then $C<\tilde C$ on $(0,\tau)$ (as follows from
the argument proving Lem\-ma~\ref
{strictMonotonicityOfIVPWRTImmigration}). Hence $\tilde C$ explodes on
$[0,\tau)$. We conclude that the branching mechanism of $\tilde Z$ is
explosive by the Ogura--Grey explosion criterion. Hence, the assumption
$\sip( Z\mbox{ reaches $\infty$ continuously} )>0$
implies that
${\Psi( 0 )}=0$ and that $\Psi$ is an explosive
branching mechanism.

On the other hand, if ${\Psi( 0 )}=0$ and $\Psi$ is
explosive, let
$\tilde\Phi=\Phi-{\Phi( 0 )}$, $Y$ be a subordinator
independent
of $X$ with Laplace exponent $\tilde\Phi$, so that sending $Y$ to
infinity at an exponential random variable with parameter ${\Phi
( 0 )}$ (independent of both $X$ and $Y$) leaves us with a subordinator
with Laplace exponent $\Phi$ independent of $X$. Let $C^1$ be a
solution to ${\ivp( x/2+X,0 )}$ and $C^2$ be a solution
to ${\ivp( x+X,Y )}$ so that $C^1\leq C^2$, and hence
$C^2$ explodes if
$C^1$ does. Let $Z^i$ be the right-hand derivative of $C^i$. $Z^1$ is a
${\cbi( \Psi,0 )}$ starting at $x/2$ while $Z^2$ is a
$\cbi(\Psi,\tilde\Phi)$ started at $x$; notice that the process $Z$
obtained by
sending $Z^2$ to infinity at the same exponential as $Y$ leaves us with
a ${\cbi( \Psi,\Phi)}$. By assumption, $X$ cannot jump to
infinity and $Z^1$ explodes with positive probability. Hence, $Z^2$
explodes with positive probability and can only do so continuously. Hence,
\begin{eqnarray*}
&&\sip( Z\mbox{ reaches $\infty$ continuously} )\\
&&\quad\geq e^{-t{\Phi( 0
)}}\sip\bigl(
Z^2\mbox{ reaches $\infty$ continuously by time $t$} \bigr)
\end{eqnarray*}
and the right-hand side is positive for $t$ large enough.

\item[(3)] We also deduce that
\[
\sip( Z\mbox{ reaches $\infty$ continuously} )=1
\]
if and only if ${\Phi( 0 )}=0$ and $\sip( Z^2\mbox
{ reaches
$\infty$ continuously} )=1$. A necessary and sufficient
condition for
the latter is that, additionally, $\Phi$ is not zero. Indeed, when
$\Phi$ is not zero, then $Y_t\to\infty$ as $t\to\infty$. Since
$\Psi$ is explosive and ${\Phi( 0 )}=0$, then $\lim_{t\to\infty
}X_t=\infty$ and so Proposition \ref
{DeterministicExplosionProposition} implies that the solution to
${\ivp( X,x+Y )}$ explodes. If $\Phi=0$, then $Z^2$ is a
${\operatorname{CBI} ( \Psi,0 )}$, which cannot explode continuously
almost surely since
the probability that $Z^2$ is absorbed at zero is the probability that
$X$ goes below $-x$, which is positive.\qed
\end{longlist}
%
\noqed\end{pf*}

\section{Stability of the generalized Lamperti transformation}
\label{StabilitySection}
We now turn to the proof of Theorems~\ref{StabilityTheorem} and \ref
{ExtensionOfPitmansTheorem}, and of
Corollaries~\ref{ContinuityOfCBILawsCorollary} and
\ref{LimitTheoremGWI}, which summarize the stability theory for
${\ivp( f,g )}$.

\subsection{Proof of the analytic assertions}
\label{ProofOfAnalyticAssertionsSubsection}
In order to compare the initial value problem ${\ivp( f,g
)}$ with
the functional inequality (\ref{IVPWithInequalities}), we now
construct an example of an admissible breadth-first pair $ (
f,g )$
such that ${\ivp( f,g )}$\vadjust{\goodbreak} has an unique solution, but
(\ref
{IVPWithInequalities}) has at least two. Indeed, consider $g=0$, and take
\[
{f ( x )}= \cases{ \sqrt{1-x}, &\quad if $x<1$,
\cr
1, &\quad if $x\geq1$.}
\]
Then ${\ivp( f,g )}$ has a unique solution, by Theorem
\ref
{ExistenceTheorem}, since $f$ is strictly positive. The solution is the
function $c$ given by
\[
{c ( t )}= \cases{ t-t^2/4, &\quad if $t\leq2$,
\cr
c(2)+t-2, &\quad if $t
\geq2$.}
\]
Since $c$ is strictly increasing, it also solves (\ref
{IVPWithInequalities}). However, the function
\[
{\tilde c ( t )}= \cases{ {c ( t )}, &\quad if $t\leq2$,
\cr
c(2), &\quad if $t\geq2$,}
\]
is also a solution to (\ref{IVPWithInequalities}). Hence, the
assumption of Theorem~\ref{StabilityTheorem} is stronger than just
uniqueness of ${\ivp( f,g )}$ although related (as seen by
comparing Theorem~\ref{ExistenceTheorem} and Proposition \ref
{UniquenessForIVPWithInequalitiesProposition}).

%

We start with a proof of Proposition \ref
{UniquenessForIVPWithInequalitiesProposition}.
\begin{pf*}{Proof of Proposition \ref
{UniquenessForIVPWithInequalitiesProposition}}
Let $c$ be any nondecreasing solution to
\[
\int_s^t {f_-\circ c ( r )}+{g ( r )} \,dr\leq
{c ( t )}-{c ( s )}\leq\int_s^t {f\circ c ( r
)}+{g ( r )} \,dr\qquad \mbox{for $s\leq t$}
\]
such that ${c ( 0 )}=0$. This automatically implies
continuity of $c$
and so $f\circ c+g$ is c\`adl\`ag and does not jump downwards.

Note that $c$ is strictly increasing if $f_-+{g ( 0 )}$ is strictly
positive or $g$ is strictly increasing, we have equalities in (\ref
{IVPWithInequalities}), implying that $c$ solves ${\ivp(
f,g )}$
which has a unique solution with these hypotheses. Indeed, if $f_-+{g
( 0 )}$ is a positive function, then the lower bound
integrand is
strictly positive, and so $c$ cannot have a constancy interval. If on
the other hand $g$ is strictly increasing, note first that the
nondecreasing character of $c$ implies, through (\ref
{IVPWithInequalities}), that $f\circ c+g$ is nonnegative (first almost
everywhere, but then everywhere since it is c\`adl\`ag). Also, $f\circ
c+g$ can only reach zero continuously since it lacks negative jumps. If
$ c$ had a constancy interval $[s,t]$ with $s<t$,
there would exist $r\in(s,t)$ such that
\[
{f_-\circ c ( s )}+{g ( s )}={f_-\circ c ( s )}+{g ( r )}=0,
\]
which implies that $g$ has a constancy interval on $[0,t]$, a
contradiction. Hence, $c$~has no constancy intervals.

When $g$ is a constant and $f_-+g$ is absorbed at zero, then also $f+g$
is absorbed at zero and at the same time. Hence, $c$ is strictly
increasing until it is absorbed, so that again both bounds for the
increments of $c$ are equal. Then $c$ solves ${\ivp( f,g
)}$ which
has a unique solution under this hypothesis.
\end{pf*}

We now continue with a proof of Theorem~\ref{StabilityTheorem}. It is
divided in two parts: convergence of the cumulative population which is
then used to prove convergence of population profiles. The strategy is
simple: we first use the functional equations satisfied by $ (
c_n )$ to prove that $c_n\wedge K$ is uniformly bounded and
equicontinuous. Then, we pass to the limit in the functional equations
satisfied by $c_n$ to see that any subsequential limit of $c_n\wedge K$
equals $c\wedge K$. [This is where the assumption that (\ref
{IVPWithInequalities}) has an unique solution comes into play.] Having
established convergence of $c_n$ to $c$, we then verify some technical
hypotheses enabling us to apply results on continuity of composition
and addition on adequate subspaces of Skorohod space and deduce that
$f_n\circ c_n+g_n$ converges to $f\circ c+g$.
\begin{pf*}{Proof of Theorem~\ref{StabilityTheorem}, convergence of cumulative
populations}
Let $K,\eps>0$ and consider the sequence $c_n\wedge K$ consisting of
nondecreasing functions with c\`adl\`ag right-hand derivatives. Since
\begin{eqnarray*}
0&\leq&{D_+ c_n\wedge K ( t )} \\
&=& \si_{{c_n ( t )}\leq K}\times\cases{
\displaystyle \bigl[{f_n\circ c_n \bigl( \lfloor t/\sigma_n
\rfloor\sigma_n \bigr)}+{g_n ( t )}\bigr]^+, &\quad if $
\sigma_n>0$,
\vspace*{2pt}\cr
\displaystyle {f_n\circ c_n ( t
)}+{g_n ( t )}, &\quad if $\sigma_n=0$,}
\\
& \leq&\sup_{y\leq K} {f ( y )}+{g ( t )}+\eps
\end{eqnarray*}
for $n$ large enough (by the convergence of $f_n\to f$ on $[0,K]$ with
the $J_1$ topology and $g_n\to g$ on $[0,t]$ with the $J_1$ topology),
we see that the sequence $c_n\wedge K$ is uniformly bounded and
equicontinuous on compacts. To prove convergence of \mbox{$c_n\wedge K$}
(uniformly on compact sets), it is enough to prove by Arzel\`a--Ascoli
that any subsequential limit is the same. Let $t>0$ and $c_{n_k}\wedge
K$ be a uniformly convergent subsequence on $[0,t]$. Denote by $\tilde
c$ its uniform limit, which is then nondecreasing. If $s\in[0,t]$ is
such that ${\tilde c ( s )}<x$, then ${c_{n_k} (
s )}<x$ for $k$
large enough.
Since $f$ has no negative jumps, then
\[
\liminf_{x\to y}{f ( x )}={f_- ( y )} \quad\mbox{and}\quad
\limsup_{x\to y}{f
( x )}={f ( y )}
\]
so that
\[
f_-\circ\tilde c\leq\liminf f_{n_k}\circ c_{n_k}\leq\limsup
f_{n_k}\circ c_{n_k}\leq f\circ\tilde c.
\]
By Fatou's lemma, for any $s_1\leq s_2\leq s$,
\[
\int_{s_1}^{s_2} \bigl[ {f_-\circ\tilde c ( r )}+{g
( r )} \bigr]^+ \,dr\leq{\tilde c ( s_2 )}-{\tilde c ( s_1
)}\leq\int_{s_1}^{s_2} \bigl[ {f\circ\tilde c ( r
)}+{g ( r )} \bigr]^+ \,dr.
\]
As $\tilde c$ is nondecreasing, we might remove the positive parts in
the above display and conclude, from uniqueness to (\ref
{IVPWithInequalities}), that $\tilde c=c$ on $[0,s]$. If, on the other
hand, ${\tilde c ( s )}=K$, then ${c_{n_k}\wedge K (
s )}\to K$
which implies that $c_{n_k}\wedge K\to c\wedge K$ uniformly on compact sets.

Let $\tau$ be the explosion time of $c$. If $t<\tau$, then ${c
( t )}<\infty$, and so [choosing $K>{c ( t )}$ in
the paragraph
above] we see that $c_n\to c$ uniformly on\vadjust{\goodbreak} $[0,t]$. If $t\geq\tau$,
then ${c ( t )}=\infty$, and so ${c_n ( t
)}\wedge K\to K$ for any
$K>0$. Hence ${c_n ( t )}\to\infty={c ( t )}$.
\end{pf*}
\begin{pf*}{Proof of Theorem~\ref{StabilityTheorem}, convergence of population
profiles}
Let
\[
h_n=D_+ c_n \quad\mbox{and}\quad h=D_+ c.
\]
We now prove that $h_n\to h$ in the Skorohod $J_1$ topology if the
explosion time $\tau$ is infinite. Recall that $h=f\circ c+g$ and that
\[
h_n=\cases{ f_n\circ c_n+g_n, &\quad
if $\sigma_n=0$,
\vspace*{2pt}\cr
\displaystyle \bigl[ {f_n\circ c_n
\bigl( \lfloor t/\sigma_n\rfloor\sigma_n
\bigr)}+{g_n \bigl( \lfloor t/\sigma_n\rfloor
\sigma_n \bigr)} \bigr]^+, &\quad if $\sigma_n>0$.}
\]

Assume that $\sigma_n=0$ for all $n$, the case $\sigma_n>0$ being
analogous. Then the assertion $h_n\to h$ is reduced to proving that
$f_n\circ c_n\to f\circ c$, which is related to the continuity of the
composition mapping on (adequate subspaces~of) Skorohod space, and then
deducing that $f_n\circ c_n+g_n\to f\circ c+g$, which is related to
continuity of addition on (adequate subspaces of) Skorohod space. Both
continuity assertions require conditions to hold: the convergence
$f_n\circ c_n\to f\circ c$ can be deduced from Wu
[(\citeyear{MR2479479}), Theorem 1.2] if we prove that $f$ is continuous at every point at which
$c^{-1}$ is discontinuous, and then the convergence of $f_n\circ
c_n+g_n$ will hold because of Whitt
[(\citeyear{1980Whitt}), Theorem~4.1]
since we
assumed that $f\circ c$ and $g$ do not jump at the same time. Hence,
the convergence $h_n\to h$ is reduced to proving that $f$ is continuous
at discontinuities of $c^{-1}$.

If $c$ is strictly increasing [which happens when $g$ is strictly
increasing or $f+{g ( 0 )}>0$], then $c^{-1}$ is
continuous. (This is
the most important case in the stochastic setting, since otherwise
immigration is compound Poisson, therefore piecewise constant, and one
might argue by pasting together Lamperti transforms.)


Suppose that $c$ is not strictly increasing, and let $x$ be a
discontinuity of~$c^{-1}$. Let
\[
s={c^{-1} ( x- )}<{c^{-1} ( x )}=t,
\]
so that $c=x$ on $[s,t]$ while $c<x$ on $[0,s)$ and $c>x$ on $(t,\infty
)$. Since $D_+c=f\circ c+g=0$ on $[s,t)$, we see that $g$ is constant
on $[s,t)$. We assert that
\[
\inf\bigl\{ y\geq0\dvtx{f ( y )}=-{g ( s )} \bigr\}=x.
\]
Indeed, if $f$ reached $-{g ( s )}$ at $x'<x$, there would exist
$s'<s$ such that
\[
0={f\circ c \bigl( s' \bigr)}+{g ( s )}\geq{f\circ c \bigl(
s' \bigr)}+{g \bigl( s' \bigr)}\geq0,
\]
so that actually $g$ is constant on $[s',t)$. Hence, $c$ has
spontaneous generation which implies there are at least two solutions
to ${\ivp( f,g )}$: one that is constant on $(s',t]$, and
$c$. This
contradicts the assumed uniqueness to (\ref{IVPWithInequalities}).
Since $f$ has no negative jumps and reaches the level $-{g (
s )}$ at
time $x$, then $f$ is continuous at $x$.\vadjust{\goodbreak}

Finally, we assume that the explosion time $\tau$ is finite but that
${f_n ( x )},{f ( x )}\to\infty$ as $x\to
\infty$ uniformly in
$n$ and prove that $h_n\to h$ in the uniform $J_1$ topology. Let $\eps
>0$, $d$ be a bounded metric on $[0,\infty]$ that makes it
homeomorphic to $[0,1]$, and consider $M>0$ such that ${d (
x,y )}<\eps$ if $x,y\geq M$. Let $K>0$ be such that ${f (
x )},{f_n ( x )}>M$ if $x>K$ and $n$ is large enough.
Let $T<\tau$
be such that $f$ is continuous at ${c ( T )}$ and $K<{c
( T )}$.
Then $f_n\to f$ in the usual $J_1$ topology on $[0,{c ( T
)}]$ and,
arguing as in the nonexplosive case, we see that
\[
h_n=f_n\circ c_n+g_n\to f\circ
c+g=h
\]
in the usual $J_1$ topology on $[0,T]$. Hence, there exists a sequence
$ ( \lambda_n )$ of increasing homeomorphisms of $[0,T]$ into
itself such that $h_n-h\circ\tilde\lambda_n\to0$ uniformly on
$[0,T]$. Define now $\lambda_n$ to equal $\tilde\lambda_n$ on
$[0,T]$ and the identity on $[T,\infty)$. Then $ ( \lambda_n )$
is a sequence of homeomorphisms of $[0,\infty)$ into itself which
converges uniformly to the identity, and since $K<{c ( T
)}$, then
$K<{c_n ( T )}$ eventually and so $M<h_n,h$ eventually
thanks to the
choice of $K$, so that ${d ( {h_n ( t )},{h (
t )} )}<\eps$ on
$[T,\infty)$ eventually. Hence, $h_n\to h$ in the uniform $J_1$ topology.
\end{pf*}

In order to apply Theorem~\ref{StabilityTheorem} to
Galton--Watson-type processes, we need a lemma relating the
discretization of the Lamperti transformation and scaling. Define the
scaling operators $S_{a}^b$ by
\[
{S_{a}^bf ( t )}=\frac{1}{b}{f ( at )}.
\]
Let also $c^\sigma$ be the approximation of span $\sigma$ to ${\ivp
( f,g )}$, which is the unique function satisfying
\[
{c^\sigma( t )}=\int_0^t \bigl[ {f
\circ c^\sigma\bigl( \sigma\lfloor s/\sigma\rfloor\bigr)}+{g \bigl(
\sigma\lfloor s/\sigma\rfloor\bigr)} \bigr]^+ \,ds.
\]
We shall denote ${c^\sigma( f,g )}$ to make the
dependence on $f$
and $g$ explicit in the following lemma and denote by ${h^\sigma
( f,g )}$ the right-hand derivative of ${c^\sigma(
f,g )}$.

%
\begin{lem}
\label{ScalingAndDiscretizationLemma}
We have
\[
{S_a^b c^\sigma( f,g )}={c^{\sigma/a}
\bigl( S_b^{b/a}f, S_a^{b/a}g \bigr)}
\quad\mbox{and}\quad {S_{a}^{b/a} h^\sigma( f,g
)}={h^{\sigma/a} \bigl( S_b^{b/a}f, S_a^{b/a}g
\bigr)}.
\]
\end{lem}
The proof is an elementary change of variables.
%

\subsection{Weak continuity of CBI laws}

\mbox{}

\begin{pf*}{Proof of Corollary~\ref{ContinuityOfCBILawsCorollary}}
Let $X_n$ and $X$ be spLps with Laplace exponents $\Psi_n$ and $\Psi$
and $Y_n$ and $Y$ be subordinators with Laplace exponents $\Phi_n$ and
$\Phi$ such that $X_n$ (resp., $X$) is independent of $Y_n$ (resp.,~$Y$).

The hypotheses $\Psi_n\to\Psi$ and $\Phi_n \to\Phi$ imply that
$ ( X_n,Y_n )$ converges weakly to $ ( X,Y )$ in
the Skorohod
$J_1$ topology. By Skorohod's representation theorem, we can assume
that the convergence takes place almost surely on an adequate
probability space.\vadjust{\goodbreak}

Let $Z_n$ (resp., $Z$) be the Lamperti transform of $ (
X_n,x_n+Y_n )$ [resp., $ ( X,x+Y )$]. When $X$ is nonexplosive,
Propositions~\ref{uniquenessForLevy} and \ref
{JumpingConditionProposition} and Theorem~\ref{StabilityTheorem} then
imply that $Z_n$ converges almost surely to $Z$, which is a ${\cbi
( \Psi,\Phi)}$, thanks to Theorem~\ref{CBIRepThm}.

When $X$ is explosive, let $\rho$ be a distance on $[0,\infty]$ which
makes it homeomorphic to $[0,1]$ and, for any $\eps>0$, choose $M_\eps
$ such that ${\rho( x,y )}<\eps$ if $x,y\geq M_\eps$.
Recall that
$d_\infty$ stands for the uniform $J_1$ topology. Since the $X^n\to X$
and $Y^n\to Y$ in the usual Skorohod topology as $n\to\infty$ almost
surely, then reasoning as in the proof of uniform $J_1$ convergence of
Theorem~\ref{StabilityTheorem}, we see that, for any $\eps>0$,
\[
\sip\bigl( {d_\infty\bigl( Z^n,Z \bigr)}>\eps,
X^n_s,X_s>M_\eps\mbox{ for all
}s\geq t \bigr)\to0 \qquad\mbox{as $n\to\infty$.}
\]
However, choosing $t$ and $M$ big enough, we can make
\[
\sip\bigl( X^n_s\leq M\mbox{ for some }s\geq t \bigr)
\]
arbitrarily small for all $n$ large enough, so that ${d_\infty(
Z^n,Z )}\to0$ in probability, which is enough to guarantee that
$Z^n\to Z$ weakly in the uniform $J_1$ topology. Indeed, since $X$ is
explosive, we have that ${\Psi' ( 0+ )}=-\infty$ [cf.
\citet{MR2466449}, proof of Theorem 2.2.3.2, page 95] which means that $X$
drifts to $\infty$; cf. \citet{MR1406564}, Chapter VII, Corollary
2.ii. Since the latter result implies that the negative of the infimum
of $X$ has an exponential distribution of parameter $\eta$, where
\[
\eta=\inf\bigl\{ \lambda>0\dvtx{\Psi( \lambda)}=0 \bigr\},
\]
we see that
\begin{eqnarray*}
&&
\sip( X_s\leq M\mbox{ for some }s\geq t ) \\
&&\qquad\leq \sip(
X_t\leq2M )+\sip( X_t>2M\mbox{ and }X_s\leq
M\mbox{ for some }s\geq t )
\\
&&\qquad\leq \sip( X_t\leq2M )+e^{-\eta M}.
\end{eqnarray*}
Since $X$ drifts to infinity, the term $\sip( X_t\leq2M
)$ goes to
zero as $t\to\infty$. Asymptotically, the same bounds hold for $X^n$
since $\Psi^n\to\Psi$ and hence, by convexity of~$\Psi$,
\[
\lim_{n\to\infty} \bigl( \inf\bigl\{ \lambda>0\dvtx{\Psi^n ( \lambda
)}=0 \bigr\} \bigr)=\inf\bigl\{ \lambda>0\dvtx{\Psi( \lambda)}=0 \bigr\}
=\eta.
\]
\upqed\end{pf*}
%
\subsection{A limit theorem for Galton--Watson processes with immigration}

\mbox{}

\begin{pf*}{Proof of Corollary~\ref{LimitTheoremGWI}}
By Skorohod's theorem, if $X$ and $Y$ are L\'evy processes whose
distributions at time $1$ are $\mu$ and $\nu$, then
\[
S^{n}_{a_n}X^n\to X \quad\mbox{and}\quad
S_{b_n}^{n}Y^n\to Y,
\]
where the convergence is in the $J_1$ topology. Assume first that $X$
is nonexplosive.\vadjust{\goodbreak}

We can apply Lemma~\ref{ScalingAndDiscretizationLemma} to get either
\[
S_{b_{\lfloor k_n/x\rfloor}}^{k_n/x}Z^n={h^{1/b_{\lfloor k_n/x\rfloor
}} \bigl(
S^{k_n/x}_{{k_n}b_{\lfloor k_n/x\rfloor
}/{x}}X^n,x+S^{k_n/x}_{b_{\lfloor k_n/x\rfloor}}Y^n
\bigr)}
\]
or
\[
S_{{x}a_{\lfloor{k_n/x}\rfloor}/{k_n}}^{k_n/x}Z^n={h^{
{k_n}/({xa_{\lfloor{k_n/x}\rfloor}})} \bigl(
x+S^{k_n/x}_{a_{\lfloor
k_n/x\rfloor}}X^n,S^{k_n/x}_{{x}a_{\lfloor{k_n/x}\rfloor
}/{k_n}}Y^n
\bigr)}.
\]

Let $Z$ be the unique process satisfying
\[
Z_t=x+X_{c\int_0^t Z_s \,ds}+Y_t
\]
as in Proposition~\ref{uniquenessForLevy}. If $\frac
{k_n}{x}b_{\lfloor k_n/x\rfloor}/a_{\lfloor k_n/x\rfloor}\to c\in
[0,\infty)$, we see that
\[
S_{b_{\lfloor k_n/x\rfloor}}^{k_n/x}Z^n\to Z,
\]
thanks to Propositions~\ref{uniquenessForLevy} and \ref
{JumpingConditionProposition} and Theorems~\ref{CBIRepThm} and \ref
{StabilityTheorem}.

When $\frac{k_n}{x}b_{\lfloor k_n/x\rfloor}/a_{\lfloor k_n/x\rfloor
}\to\infty$,
let $Z$ instead be the unique solution to
\[
Z_t=x+X_{\int_0^t Z_s \,ds}.
\]
Then
\[
S_{{x}a_{\lfloor{k_n/x}\rfloor}/{k_n}}^{k_n/x}Z^n\to Z.
\]

When $X$ is explosive, the arguments in the proof of Corollary \ref
{ContinuityOfCBILawsCorollary} show that, in order to obtain the stated
convergence in the uniform $J_1$ topology, it is enough to prove that
for all $M>0$,
\[
\lim_{M\to\infty}\lim_{t\to\infty}\limsup_n \sip\biggl(
\frac
{1}{n}X^n_{\lfloor s a_n\rfloor}\leq M \mbox{ for some }s\geq t
\biggr)=0.
\]
Since $X$ drifts to infinity if it is explosive, $\Psi$ has an unique
positive root which we denote $\eta$.

Let
\[
{G_n ( \lambda)}=\se\bigl( e^{-\lambda X^n_1} \bigr).
\]
Recall that since the increments of $X^n$ are bounded below by $-1$,
minus the random variable
\[
I_n=\min_{m\geq0} X^n_m
\]
has a geometric distribution with parameter $e^{-\eta_n}$ where $\eta
_n$ is the greatest nonnegative real number at which $G_n$ achieves
the value $1$; cf. Asmussen [(\citeyear{MR1978607}), Part B, Chapter VIII,
Section 5, Corollary 5.5, page 235] or the forthcoming Lemma \ref
{ConvergenceOfGeneratingFunctionsLemma}. By log-convexity of $G_n$,
$\eta_n=\inf\{ \lambda>0\dvtx{G_n ( \lambda
)}>1 \}$. If
we assume
the convergence of $n\eta_n$ to $\eta$ as $n\to\infty$, we see that
\[
\limsup_n\sip\biggl( -\frac{1}{n}I_n\geq M
\biggr)=e^{-\eta M}.
\]
We now use the Markov property to conclude that if the distribution of
$X_1$ is continuous at $M$, then
\[
\limsup_n \sip\biggl( \frac{1}{n}X^n_{\lfloor s a_n\rfloor}
\leq M \mbox{ for some }s\geq t \biggr)\leq\sip( X_t\leq2M )+\sip(
X_t\geq2M )e^{-\eta M}.
\]

To conclude, we should prove that $n\eta_n\to\eta$. This, however,
is implied by the following convergence of Laplace transforms:
\[
\se\bigl( e^{-\lambda/n X^n_{a_n}} \bigr) \to\se\bigl(
e^{-\lambda X_1} \bigr)= e^{{\Psi( \lambda
)}}.
\]
Indeed, recall that $\se( e^{-\lambda X_1} )<1$ exactly
on $(0,\eta)$
and that $\se( e^{-\lambda/ nX^n_{a_n}} )<1$ exactly on
$(0,n\eta_n)$. If we consider $\lambda<\eta$ then $\se( e^{-\lambda/
nX^n_{a_n}} )<1$ for large enough $n$, so that $\lambda\leq
n\eta_n$
for large enough $n$. This implies $\eta\leq\liminf_{n}n\eta_n$;
the upper bound is proved similarly. Convergence of Laplace transforms
is actually the condition imposed by \citet{MR2225068} to prove limit
theorems for Galton--Watson processes with immigration. That this
already follows from our hypotheses is the content of the following
lemma, which concludes the proof of Corollary~\ref{LimitTheoremGWI}.
\end{pf*}
%
%
\begin{lem}
\label{ConvergenceOfGeneratingFunctionsLemma}
Let $X^n$ be a sequence of random walks with jumps in $ \{
-1,0,\break1,\ldots\}$ satisfying the conditions of Corollary \ref
{LimitTheoremGWI}, and suppose that $X$ is not a subordinator. Then
\[
\se\bigl( e^{-\lambda/n X^n_{a_n}} \bigr)\to e^{{\Psi(
\lambda)}}
\]
for all $\lambda>0$.
\end{lem}
This is the content of Theorem 2.1 of \citet{MR0362529}; we
present a
proof using basic fluctuation theory for independent increment processes.
\begin{pf*}{Proof of Lemma~\ref{ConvergenceOfGeneratingFunctionsLemma}}
Using Skorohod's theorem again, we assume that $X^n_{\lfloor a_n\cdot
\rfloor}/n$ converges almost surely to $X$ in the Skorohod $J_1$ topology.
Also, enlarge the probability space so that it admits an exponential
random variable $R_\lambda$ of parameter $\lambda$ which is
independent of $X$ and $X^n$.

Let
\[
{G_n ( \lambda)}=\se\bigl( e^{-\lambda X^n_{1}} \bigr). %
\]
Since $X$ is not a subordinator, then $\sip( X^n_1=-1
)>0$ for large
enough $n$, and we can assume that this happens for every $n$. Hence,
${G_n ( \lambda)}\to\infty$ as $\lambda\to\infty$,
and we can define
\[
{F_n ( s )}=\inf\bigl\{ \lambda>0\dvtx{G_n ( \lambda)}>1/s
\bigr\} \qquad\mbox{for $s\in(0,1]$.}
\]
Using optional sampling at the first time $T^n_k$ at which $X^n$
reaches $-k$ for the first time, applied to the martingale
\[
e^{-\lambda X^n_m}{G_n ( \lambda)}^{-m},
\]
we obtain
\[
\se\bigl( s^{T^n_k} \bigr)=e^{-k{F_n ( s )}}
\]
for $s\in(0,1]$. Define the random variables
\[
I_\lambda=\min_{s\leq R_\lambda} X_s \quad\mbox{and}\quad
I^n_\lambda=\min_{s\leq R_\lambda} \frac{1}{n}X^n_{\lfloor a_n
s\rfloor}.
\]
Since $\lfloor a_nR_\lambda\rfloor$ has a geometric distribution of
parameter $e^{-\lambda/a_n}$, it follows that
\[
\sip\bigl( -nI^n_\lambda\geq k \bigr)=\sip\bigl(
T^n_k<\lfloor a_n R_\lambda\rfloor
\bigr)=\se\bigl( e^{-\lambda/a_n T^n_k} \bigr)=e^{-k{F_n ( e^{-\lambda
/ a_n} )}}
\]
so that $-nI^n_\lambda$ has a geometric distribution. Also, from
Corollary 2 in Bertoin
[(\citeyear{MR1406564}), Chapter VII], $I_\lambda$ has an
exponential distribution of parameter ${\tilde\Phi( \lambda
)}$ where
\[
{\tilde\Phi( \lambda)}=\inf\bigl\{ \tilde\lambda>0\dvtx{\Psi( \tilde
\lambda)}>
\lambda\bigr\}.
\]
However, since $X$ does not jump almost surely at $R_\lambda$ and the
minimum is a continuous functional on Skorohod space (on the interval
$[0,R_\lambda]$), we see that $I^n_\lambda$ converges weakly to
$I_\lambda$. This implies
\[
n{F_n \bigl( e^{-\lambda/a_n} \bigr)}\to{\tilde\Phi( \lambda)},
\]
and by passing to inverses, we get
\[
{G_n ( \lambda/n )}^{a_n}\to e^{{\Psi( \lambda
)}}
\]
for $\lambda>{\tilde\Phi( 0 )}$.

Finally, if $\lambda\in(0,{\tilde\Phi( 0 )}]$, pick
$p>1$ such
that $p\lambda>{\tilde\Phi( 0 )}$; we have just proved
that the sequence
\[
{G_n ( p\lambda/n )}^{a_n},\qquad n\geq1,
\]
and being convergent, it is bounded. Hence the sequence
\[
e^{-\lambda/n X^n_{a_n}},\qquad n\geq1,
\]
is bounded in $L_p$ and converges weakly to $e^{-\lambda X_1}$. We then get
\[
{G_n ( \lambda/n )}^{a_n}=\se\bigl( e^{-\lambda/n X^n_{a_n}}
\bigr)\to\se\bigl( e^{-\lambda X_1} \bigr)=e^{{\Psi(
\lambda)}}.
\]
\upqed\end{pf*}
%
\subsection{A limit theorem for conditioned Galton--Watson processes}

\mbox{}

\begin{pf*}{Proof of Theorem~\ref{ExtensionOfPitmansTheorem}}
Let $Z^{n}$ be a Galton--Watson process with critical offspring law
$\mu$ such that $Z^{n}_0=k_n$ and is conditioned on $\sum_{i=1}^\infty
Z^{n}_i=n$. Then, $Z^{n}$ has the law of the discrete
Lamperti transformation of the $n$ steps of a random walk with jump
distribution $\tilde\mu$ (the shifted reproduction law) which starts
at $0$ and is conditioned to reach $-k_n$ in $n$ steps; call the latter
process $X^{n}$, so that
\[
Z^{n}={h^1 \bigl( k_n+X^{n},0
\bigr)}.\vadjust{\goodbreak}
\]

Thanks to \citet{MR2534486}, if $k_n/a_n\to l$, then
\[
S_{n}^{a_n}X^{n}\to F^l.
\]
Thanks to Lemma~\ref{ScalingAndDiscretizationLemma}, we see that
\[
S_{n/a_n}^{a_n}Z^{n}={h^{a_n/n} \bigl(
S_{n}^{a_n}X^{n},0 \bigr)}.
\]
Let $\alpha\in(1,2]$ be the index of the stable process in the
statement of Theorem~\ref{ExtensionOfPitmansTheorem}, and recall that
$a_n$ is of the form $n^{1/\alpha}{L ( n )}$ where $L$ is
a slowly
varying function, so that $a_n={o ( n )}$. Since $F^l$ is
absorbed at
zero [as is easily seen by the pathwise construction of $F^l$ by
\citet{MR2534486}, Theorem~4.3], then Proposition~\ref
{UniquenessForIVPWithInequalitiesProposition} guarantees that the
Lamperti transform $Z$ of $F^l$ is the unique process which satisfies
\[
\int_s^t F^l_{\int_0^r Z_u \,du-}\leq
\int_s^t Z_r \,dr\leq\int
_s^t F^l_{\int_0^r Z_u \,du}.
\]
Theorem~\ref{StabilityTheorem} implies that
\[
S_{n/a_n}^{a_n}Z^{n}\to Z.
\]
\upqed\end{pf*}

\section*{Acknowledgments}

G. Uribe Bravo would like to thank Jim Pitman for his constant
encouragement as a postdoctoral supervisor and stimulating
conversations around conditioned Galton--Watson processes. We would
like to thank the referee for a conscientious and detailed analysis of
our work which helped us remove an important number of misprints and
clarify some obscure points.


%

\printaddresses

\end{document}